\documentclass[12pt]{elsarticle}
\usepackage{amssymb}
\usepackage{amsmath}
\usepackage{amsfonts,dsfont}
\usepackage{algorithm}
\usepackage{algorithmic}
\usepackage{graphicx}
\makeatletter
\def\ps@pprintTitle{%
 \let\@oddhead\@empty
 \let\@evenhead\@empty
 \def\@oddfoot{}%
 \let\@evenfoot\@oddfoot}
\makeatother

\usepackage{tikz}
\usetikzlibrary{positioning,chains,fit,shapes,calc}
\usetikzlibrary{decorations.pathreplacing}
\usetikzlibrary{arrows,positioning}
\usepackage{mathtools}

\usepackage{hyperref}
\newcommand{\arxiv}[1]{\href{http://www.arXiv.org/abs/#1}{arXiv:#1}}
\usepackage{tikz}
\usetikzlibrary{arrows}
\newcommand {\junk}[1]{}
\usepackage{enumerate}
\usepackage{appendix}

\usepackage{amsthm}

\newtheorem{theorem}{Theorem}[section]

\newtheorem{proposition}[theorem]{Proposition}

\theoremstyle{definition}

\newtheorem{example}[theorem]{Example}

\newtheorem{problem}{Problem}[section]

\def\R{{\mathbb R}}
\def\Nat{{\mathbb N}}
\def\Z{{\mathbb Z}}
\def\Rmax{\R_{\max}}
\usepackage{csquotes}
\def\Rmax{\mathbb{R}_{\max}}
\def\Rmin{\mathbb{R}_{\min}}
\def\cS{\mathcal{S}}
\def\cT{\mathcal{T}}

\begin{document}

\begin{frontmatter}



\setcounter{footnote}{1}
\title{Tropical pseudolinear and pseudoquadratic optimization as parametric mean-payoff games}

\tnotetext[t1]{
The work of S. Sergeev was supported by EPSRC grant EP/P019676/1. This work originated as an M.Sci. project of J. Parsons under the supervision of S. Sergeev and was developed during the research visit of the third author to Birmingham, supported by National Natural Science Foundation of China, grant no. 11901486}

\author[rvt1]{Jamie Parsons}
\ead{jamie.parsons@hotmail.co.uk}

\author[rvt1]{Serge{\u\i} Sergeev\corref{cor}}
\ead{s.sergeev@bham.ac.uk}

\author[rvt2]{Huili Wang}
\ead{huiliwang77@163.com}  

\address[rvt1]{University of Birmingham, School of Mathematics, 
Edgbaston B15 2TT, UK.}
\address[rvt2]{School of Sciences,
Southwest Petroleum University, Chengdu, Sichuan 610500, China
}

\cortext[cor]{Corresponding author. Email: s.sergeev@bham.ac.uk}

\begin{abstract} We apply an approach based on parametric mean-payoff games to develop bisection and Newton schemes for solving problems of tropical pseudolinear and pseudoquadratic optimization with general two-sided constraints.   
\end{abstract}


\journal{????????????????}






\end{frontmatter}


\section{Introduction}


Tropical linear algebra is a relatively new area of mathematics, having only been studied in depth since around 1960. Cuninghame-Green was one of the pioneers in the area, with many of his results being presented in \cite{MiniMax}. Since then, a number of mathematicians have further developed the theory and applications of tropical linearity, and it has enjoyed prominence and influence in several mathematical areas such as linear algebra, algebraic geometry and combinatorial optimization. Tropical linear algebra gives us the ability to write seemingly non-linear problems in a linear fashion using tropical linear operators, allowing an algebraic encoding of many combinatorial problems and providing a framework with which some discrete problems can be modelled \cite{baccelli,MainBook}.


Optimization problems over tropical linear algebra have been studied since 1970's and 1980's: see, e.g., K.~Zimmermann~\cite{KZim} and U. Zimmermann~\cite{UZim:81}. A crucial step was made by Butkovi\v{c}~\cite{MainBook}, who developed a bisection scheme for solving tropical linear programming.  Krivulin~\cite{KrivPseudolin-12,KrivPseudolin-14, KrivOurVolume-14, KrivSurvey-14, KrivUnconstr-15, KrivPseudoquad-15, KrivLocation-17} made a big contribution to the area by solving a multitude of what we term as tropical pseudolinear and tropical pseudoquadratic optimization problems, albeit with special constraints.  The connection between tropical linear algebra and mean payoff games was discovered by Akian, Gaubert and Guterman~\cite{TropicalMPG}, and this connection was further applied to tropical linear-fractional programming by Gaubert et al.~\cite{BiNewt}. The present paper aims to develop the tropical pseudolinear and pseudoquadratic optimization by considering a more general form of constraints and applying to it the connection to mean-payoff games discovered in~\cite{TropicalMPG}, following the approach of~\cite{BiNewt}.

We now introduce the tropical linear algebra formally.
Tropical semiring is the set $ \Rmax = \R \cup \{- \infty\}$ equipped with the operations $(\oplus, \otimes)$ defined by
$a \oplus b = \max(a,b)$ and
$a \otimes b = a + b$. 
Algebraically, ($\Rmax, \oplus, \otimes$) is a commutative idempotent semifield, meaning in particular that $\oplus$ and $\otimes$ are commutative, associative and satisfy the distributive law. Note that the $\oplus$ identity is $-\infty$ and the $\otimes$ identity is $0$. Many of the tools and operations used in usual linear algebra can be applied in tropical linear algebra.  An important difference with rings is that $(\mathbb{R}\cup
\{-\infty\}, \oplus)$ does not form a group and thus there is no straightforward subtraction in tropical algebra. Symmetrization over tropical algebra has been known for a long time~\cite{baccelli}, being a kind of substitute for this shortcoming. Also we can use the idempotent property of $\oplus$: for all $x \in \Rmax$ $x \oplus x  = x$.

The definition of $\oplus$ and $\otimes$ is then extended to include matrices and vectors. 
  Suppose that $A=(a_{ij})$, $B=(b_{ij})$, $C=(c_{ij})$, where  $a_{ij}$, $b_{ij}$, $c_{ij}\in \Rmax$. When $A$ and $B$ are of the same size, $C = A \oplus B$ if $c_{ij} = a_{ij} \oplus b_{ij}$ for all $i,j$. 
 For matrices $A$ and $B$ of compatible size, $ D = A \otimes B $ if 
$d_{ij} =\bigoplus a_{ik} \otimes b_{kj}= \text{max}_{k} (a_{ik} + b_{kj})$.

The unit matrix, $I$, is defined as a square matrix whose diagonal entries are $0$, with all off-diagonal entries being $-\infty$. 

For $a\in\R\cup\{-\infty\}\cup \{+\infty\}$ we
define its conjugate $a^-$ as $-a$ if $a\in\R$, $+\infty$
if $a=-\infty$, and $-\infty$ if $a=+\infty$. This definition is easily extended to vectors: for a column vector $x=(x_i)$ its conjugate is defined as a row vector $x^-=(x_i^-)$. 

The ideas and methods of tropical linear algebra and mean-payoff games will help us to solve the following main problems considered in this paper.

\begin{problem}[Pseudolinear optimization with two-sided constraints]
\label{prob:mainpseudolin}
Given $p\in \Rmax^n$, $q\in(\R\cup\{+\infty\})^n$, $U,V\in\Rmax^{m\times n}$, $b,d\in\Rmax^m$, find 
\begin{equation*}
    \begin{aligned}
& \underset{x\in\R^n}{\min}
& & x^-\otimes p \oplus q^-\otimes x\\
& \text{subject to}
& & U\otimes  x \oplus b \leq V\otimes x \oplus d.
    \end{aligned}
\end{equation*}
and at least one finite $x\in\R^n$ that attains the minimum.
\end{problem}

\begin{problem}[Pseudoquadratic optimization with two-sided constraints]
\label{prob:main}
Given $C\in\Rmax^{n\times n}$, $p\in \Rmax^n$, $q\in(\R\cup\{+\infty\})^n$, $U,V\in\Rmax^{m\times n}$, $b,d\in\Rmax^m$, find 
\begin{equation*}
    \begin{aligned}
& \underset{x\in\R^n}{\min}
& & x^-\otimes C\otimes x \oplus x^-\otimes p \oplus q^-\otimes x\\
& \text{subject to}
& & U\otimes  x \oplus b \leq V\otimes x \oplus d.
    \end{aligned}
\end{equation*}
and at least one finite $x\in\R^n$ that attains the minimum.
\end{problem}

To our knowledge, these problems have not been considered before, and their algorithmic solution based on parametric mean-payoff games, inspired by the solution given by Gaubert et al.~\cite{BiNewt} for the tropical linear-fractional programming, will be suggested in the present paper for the first time.
The problems with the same objective function as Problem~\ref{prob:mainpseudolin} but more special constraints were studied by Krivulin~\cite{KrivPseudolin-14, KrivLocation-17} 
see also \cite{KrivPseudolin-12} and  \cite{KrivZim-13} where the objective fuction is strongly related to the one in~Problem~\ref{prob:mainpseudolin}. Problems with similar objective functions as in Problem~\ref{prob:main}
(in particular, containing the pseudoquadratic term $x^-\otimes C\otimes x$) were also studied by Krivulin~\cite{KrivOurVolume-14,KrivPseudoquad-15,KrivPseudoquad-17}. However, the problems considered in these works have more special constraints than Problems~\ref{prob:mainpseudolin} and ~\ref{prob:main}. In return for less general constraints, a comprehensive and concise description of both optimal value and the whole solution set is offered in all of the above references.

In contrast to that approach, we are not interested to describe all solutions. Although it is possible to achieve such a description following, for example, the double description algorithm by Allamigeon et al.~\cite{AGG10}, it is much more complicated than in the case of special systems of constraints, and we are not intending to do it here.

As in the case of the tropical linear or tropical linear-fractional programming~\cite{MainBook, BiNewt}, there is a clear geometric motivation to consider the problems posed above, since the constraint $U\otimes x \oplus b \leq V\otimes x \oplus d$ is a general form of two-sided tropical affine constraints, that is, such constraints that describe the tropical polyhedra.

We now show how this type of pseudo-quadratic optimization problem may arise in practice, by considering an application to project scheduling, similar to the one described by Krivulin~\cite{KrivOurVolume-14,KrivPseudoquad-15,KrivPseudoquad-17}, but also including disjunctive constraints, which were not considered in those works.

Suppose a particular project involves a set of $n$ activities that have to be completed. Each activity has an initiation time $x_i$ and a completion time $y_i$. We define the time it takes to complete each activity to be its flow time, given by $y_i-x_i$. 

We have ``start-to-finish'' constraints, that activities cannot be completed until specified times have elapsed after the initiation of other constraints, and that the activities are completed as soon as possible within these constraints. This means that $x$ and $y$
 should satisfy the inequalities $x_j+c_{ij}\leq y_i$ for all $i$ and $j$, where $c_{ij}=-\infty$ if there is no such constraint for some $i$
 and $j$. For the activity $j$ to start as soon as possible, we should have $\max_j(c_{ij}+x_j)=y_i$, and hence we obtain 
 $y_i-x_i=-x_i+\max_j(c_{ij}+x_j)$ for the flow time of task $i$. 
We will be also interested to minimise $p_i-x_i$ and $x_i-q_i$ for each task $i$, which means that we would like task $i$ not to delay too much after $q_i$ and not to start too much in advance before $p_i$. Note that in the project scheduling practice, one is interested either in minimising the greatest flow times or in minimising the greatest of the above mentioned time differences. However, it will be more mathematically convenient for us to consider them together and pose the objective to minimise the greatest of all flow times and the time differences $p_i-x_i$ and $x_i-q_i$. Thus we obtain the objective
 $$
 \min_{x\in\R^n} x^-\otimes C\otimes x\oplus x^-\otimes p\oplus q^-\otimes x.
 $$
This objective allows us to consider both kinds of objectives at the same time, as we can allow some (or even all) entries of $C$, $p$ or $q^-$ to be $-\infty$.

The starting times of the tasks may be subject to more constraints: in the simplest case it may be required that $b_i\leq x_i\leq d_i$ for some $i$ and some $b_i,d_i\in\R$, or that $u_{ij}+x_j\leq x_i$ for some $i$ and $j$ and $u_{ij}\in\R$. However, we may also have {\em disjunctive constraints}, where $b_i\leq v_{ij}+x_j$ , for some $i$ and with $b_i,\,v_{ij}\in\R$, should hold at least for one $j$. Similarly, we may have that 
$\max_k (u_{ik}+x_{k})\leq v_{ij}+x_j$, for some $i$ with $u_{ik}\in\Rmax$ and $v_{ij}\in\R$, should hold for at least one $j$. This motivates considering the above objective function with 
constraint in the form of two-sided tropical affine inequality
$$
U\otimes x\oplus b\leq V\otimes x\oplus d,
$$
which can capture all the above mentioned constraints.

The rest of the paper will be organized as follows. In Section~\ref{s:prel} we recall the basic knowledge and facts about tropical linear algebra and mean-payoff games that are needed for this paper. In Section~\ref{s:MPG} we formulate  Problem~\ref{prob:mainpseudolin} in terms of the associated parametric mean-payoff game and present the certificates of optimality and unboundedness in terms of this game. In Sections~\ref{s:bisection} and~\ref{s:Newt} we develop the bisection and Newton schemes for solving Problem~\ref{prob:mainpseudolin}. Two versions of the Newton scheme are given: one for the case of real valued data and the other for the case of integer data. In Section~\ref{s:appl} we give an example and describe the numerical experiments, which we conducted upon implementing the Newton and bisection schemes in MATLAB. Finally, in Section~\ref{s:pseudoquad} we explain how most of our results, including the Newton and bisection schemes,  extend to pseudoquadratic programming (Problem~\ref{prob:main}). However, efficient implementation of these schemes for this problem will be developed in another publication.      

\section{Preliminaries}
\label{s:prel}

\subsection{Tropical linear algebra}

Some of the main concepts of tropical linear algebra come from combinatorial optimization~\cite{MainBook}.

 With a matrix $A\in(\R\cup\{-\infty\})^{n\times n}$ we can associate a weighted digraph $D_A=(N,E,w)$. It has the set of nodes $N=\{1,\ldots,n\}$, set of arcs 
 $E\subseteq N\times N$ such that $(i,j)\in E$ if and only if $a_{ij}\neq -\infty$, and weight function 
 $w: E\rightarrow R$ defined by $w(i,j)=a_{ij}$.
 Vice versa, if we are given a weighted digraph $D$ with $n$ nodes, we can associate with it a matrix $A_D\in(\R\cup \{-\infty\})^{n\times n}$, along the same lines.

One important player in tropical linear algebra is the maximum cycle mean, $\rho(A)$. Given a matrix $A \in (\R\cup -\infty)^{n\times n}$, we define
\begin{equation*}
    \rho(A) = \underset{\sigma}{\max}\frac{w(A,\sigma)}{l(\sigma)}.
\end{equation*}
\noindent Here, $\sigma$ denotes an elementary cycle. Also if we have 
$\sigma=(i_1\ldots i_ki_1)$ then $w(A,\sigma)=a_{i_1i_2}\otimes a_{i_2i_3}\otimes\ldots\otimes a_{i_ki_1}$ 
is the weight of $\sigma=(i_1\ldots i_ki_1)$ and $l(\sigma)=k$ is the length of the cycle $\sigma$.
Note that 
$\rho(A)$ is well defined for any matrix $A$, and $\rho(A) = -\infty$ if and only if $D_A$ is acyclic. 

The importance of $\rho(A)$ is due to the following. 
On the one hand, in the usual linear algebra, the behaviour of some iterative algorithms crucially depends on 
the existence and properties of $(I-A)^{-1}$. In the tropical linear algebra, this is replaced with 
the Kleene star defined by the formal series
$$
A^*=I\oplus A\oplus A^2\oplus \ldots.
$$
This series converges and is equal to $I\oplus A\oplus\ldots \oplus A^{n-1}$ if and only if $\rho(A)\leq 0$,
that is, if and only if there is no cycle in $D_A$ with a positive weight.

On the other hand, $\rho(A)$ is crucial for the tropical eigenvector-eigenvalue problem, or spectral problem in tropical linear algebra, which is the problem of finding tropical eigenvalue $\lambda$ and
tropical eigenvector $x \neq -\infty$ such that
 $ A \otimes x = \lambda \otimes x.$
The connection between $\rho(A)$ and the tropical spectral problem is that
the greatest tropical eigenvalue of any square matrix over $\Rmax$ is equal to 
$\rho(A)$, see~\cite{baccelli, MainBook}.




\subsection{Dual Operators and Conjugates}

In tropical matrix algebra, matrix inverses exist only for a very special case of diagonal and monomial matrices. However, we can overcome some of the difficulties this poses by defining  conjugates.

 Dual operations are defined using the min plus algebra,
which is the set $\Rmin=\R \cup \{+ \infty\}$ equipped with operations $(\oplus', \otimes')$ defined by 
     $a\oplus' b= \min(a,b)$ and
     $a\otimes' b = a + b$ for all $a,b\in\Rmin$.

Recalling the definition of scalar conjugate 
$a\to a^-$, the conjugate of a matrix $A$ with entries in $\R\cup \{-\infty\}$, denoted $A^{\sharp}$, can be then defined by
\begin{equation*}
    (A^{\sharp})_{ij} = a_{ji}^{-}\quad\forall i,j.
\end{equation*}
Note that as $A$ has entries in $\Rmax$ and can be used to perform max-plus multiplication, $A^{\sharp}$ has entries in $\R\cup\{+\infty\}$ and can be used to perform min-plus multiplication.

Below we will use the following property of 
scalar conjugates: for $a, x\in\Rmax$ 
and $b\in\R\cup\{+\infty\}$, we have
$a \otimes x \leq b$ if and only if $x \leq a^- \otimes' b$. This property can be easily extended to matrices:

\begin{proposition} [e.g., \cite{MainBook} Theorem 1.6.25]
\label{p:equiv}
Let $A \in \Rmax^{m \times n}, x\in \Rmax^n$ and $b \in (\R\cup\{+\infty\})^m$. Then
\begin{equation*}
A \otimes x \leq b\quad \text{if and only if}\quad x \leq A^{\sharp} \otimes' b.
\end{equation*}
\end{proposition}

The scalar and matrix conjugations discussed above were introduced by Cuninghame-Green~\cite{MiniMax} and as residuations of max-plus scalar and matrix multiplication by Baccelli et al.~\cite{baccelli}. In particular, we use the notation $\sharp$ following Baccelli et al.~\cite{baccelli} to emphasize the duality between the operator $A\otimes$ and its residuation $A^{\sharp}\otimes'$, but we prefer to use a more intuitive notation $a^-$ and $x^-$ for scalars and vectors.

\subsection{Two-sided systems and min-max functions}


The following is an obvious corollary of Proposition~\ref{p:equiv}:
\begin{equation*}
A \otimes x \leq B \otimes x\quad\text{if and only if}\quad x \leq A^{\#} \otimes'(B \otimes x).
\end{equation*}

Here we assume that there is a finite entry in each row of $B$ and each column of $A$. With this assumption, the inequality  $x \leq A^{\#} \otimes'(B \otimes x)$ can be written 
as the following system of inequalities:
\begin{equation}
\label{e:twosided-system}
    x_{j} \leq \underset{k\colon a_{kj}\in\R}{\min}(-a_{kj}+\underset{l\colon b_{kl}\in\R}{\max}(b_{kl}+x_{l})) \quad \forall j=1,..,n. 
\end{equation}

The expression on the right-hand side of~\eqref{e:twosided-system} can be written as 
the component of a
{\em min-max function} $f(x)= A^{\sharp} \otimes'(B \otimes x)$: 
\begin{equation*}
    f_j(x) = \underset{k\colon  a_{kj}\in\R}{\min} (-a_{kj}+\underset{l\colon b_{kl}\in\R}{\max}(b_{kl}+x_l)).
\end{equation*}
This function belongs to the class of topical functions investigated by Gaubert and Gunawardena~\cite{GG-04}. For the present paper, we will only need that the {\em cycle-time vector} of $f$: 
\begin{equation}
\label{e:cycletime}
    \chi(f) = \underset{k \rightarrow \infty}{\lim} \frac{f^{k}(0)}{k},
\end{equation}
exists and is well-defined~\cite{GG-98}.
Here, $0$ denotes the vector with all components equal to $0$.

In what follows, 
expression of the form $A^{\sharp}B$ will denote the min-max function $x\mapsto A^{\sharp}\otimes'(B\otimes x)$. In particular, note that, by default, the min-plus multiplication is always to be expected after the $\sharp$ sign. Cycle-time vectors of such functions $\chi(A^{\sharp}B)$ and their individual components $\chi_j(A^{\sharp}B)$ play a crucial role in the mean-payoff game approach to tropical optimization problems.

\subsection{Mean Payoff Games}

Consider the following zero-sum two player sequential game defined by two $m \times n$ matrices $A = (a_{ij})$ and $B = (b_{kl})$ over $\Rmax$. The game is played on the weighted directed graph $(V,E,w)$ where the set of nodes $V=[m]\cup [n]$ is the union of $m$ nodes corresponding to the rows of the system $A\otimes x\leq B\otimes x$ and $n$ nodes corresponding to the variables. The arc set $E$ contains 1) the arcs $(k,l)$ for which $b_{kl}\in\R$ and 
2) the arcs $(j,i)$ for which $a_{ij}\in\R$, and these arcs are weighted by $w(k,l)=b_{kl}$ and 
$w(j,i)=-a_{ij}$ respectively. The nodes in $[m]$ are the nodes at which player Max is active, and
the nodes in $[n]$ are the nodes at which player Min is active. The game starts at a node $j\in [n]$ of Min, and first Min
chooses to move a pawn to some node $i\in [m]$ of Max, via a weighted arc for which $a_{ij}\neq -\infty$. 
In doing so, player Max receives the payment $-a_{ij}$ from Min. Then player Max chooses to move to some $l\in [n]$ for which $b_{il}\neq -\infty$. Player Max then receives a payment of $b_{il}$ from player Min, and then the game proceeds
sequentially in turns. 
Player Max then wishes to maximize the average reward per turn of the infinite trajectory thus created, and
Min wishes to minimize it.

 Ehrenfeucht and Mycielski~\cite{MPGEarly}, who  introduced this game, also showed that 
it is equivalent to a finite game, which ends as soon as the trajectory forms a cycle. In this finite game, we also assume 
that Max and Min play according to some positional strategies: they choose in advance some mapping $\sigma\colon [m]\to [n]$ and $\tau\colon [n]\to [m]$ such that Max always moves the pawn from $i\in [m]$ to $\sigma(i)\in [n]$ and 
Min moves the pawn from $j\in [n]$ to $\tau(j)\in [m]$, so that $b_{i\sigma(i)}\neq -\infty$ and 
$a_{\tau(j)j}\neq -\infty$ for all $i\in [m]$ and $j\in [n]$.  The two positional strategies form a digraph, whose arc set is a subset of the edges in the original mean payoff game, with each node having a unique outgoing arc, either $(i, \sigma(i))$ or $(j, \tau(j))$ (see, for example, \cite{BiNewt}). The mean weight per turn of the first cycle formed by the trajectory is then the reward of player Max, and the game is called winning for Max if this reward is nonnegative. Max then wants to maximize his reward by choosing a better positional strategy, and Min counterplays to minimize her loss by choosing her positional strategy.

Note that the assumption written above guarantees that Max and Min have at least one positional strategy that they can use.  

The mathematical expression for the reward of Max in the finite mean-payoff game starting at node $j$ of Min, given that Max employs positional strategy $\sigma$ and Min employs positional strategy $\tau$, is
\begin{equation*}
    \Phi_{A,B}(j, \tau, \sigma) = 
\frac{1}{k}\left(\sum_{t=1}^{k} \left(b_{r_t s_t}-a_{r_{t+1}s_t}\right)\right),
\end{equation*}
Here $r_1s_1r_2s_2\ldots r_ks_kr_1$ is the unique cycle formed by the trajectory, which starts at a 
node $j$ of Min and develops according to the positional strategies $\sigma$ and $\tau$. It is assumed that $r_{k+1}=r_1$

Ehrenfeucht and Mycielski~\cite{MPGEarly} showed that such mean-payoff game (both the infinite and the finite version of it) has value in pure strategies. The following can be deduced from~\cite{MPGEarly} and Gaubert and Gunawardena ~\cite{GG-98}.

\begin{proposition}[\cite{MPGEarly,GG-98}]
\label{p:MPG-value}
For a mean payoff game, where $A,B \in \mathbb{R}_\text{max}^{m \times n}$ both have a finite entry in each row and column, there exist $\tau^*$, $\sigma^*$ such that for all $j=1,\ldots,n$ \begin{equation}
\label{ValueMPG}
\min_{\tau\in\cT} \Phi_{A,B}(j,\tau,\sigma^*)=\Phi_{A,B}(j,\tau^*,\sigma^*)=\max_{\sigma\in\cS} \Phi_{A,B}(j,\tau^*,\sigma),
\end{equation}
where $\cT$ and $\cS$ denote, respectively, the sets of positional strategies available to Min and to Max. Moreover, $\Phi_{A,B}(j,\tau^*,\sigma^*)=\chi_j(A^{\sharp}B)$ for all $j$ (where $\chi_j(A^{\sharp}B)$ is the $j$th component of the cycle-time vector defined in~\eqref{e:cycletime}).
\end{proposition}

Finally, the relationship between mean-payoff games and 
the solvability of tropical two-sided systems can be summarized in the following result.

\begin{proposition} [\cite{TropicalMPG} Theorem 3.1]
\label{p:tropicalMPG}
For a mean payoff game, where $A,B \in \mathbb{R}_\text{max}^{m \times n}$ both have a finite entry in each row and column and $\sigma^*$ and $\tau^*$ are equilibrium strategies, we have
$\Phi_{A,B}(j,\sigma^*,\tau^*)(=\chi_j(A^{\sharp}B)) \geq 0$ if and only if there exists $x \in \mathbb{R}^{n}_{\text{max}}$ such that
$A\otimes x \leq B\otimes x$ and
$x_j > -\infty$.
\end{proposition}

Let us now consider the following example of a mean payoff game.

\begin{example}
 Let the matrices $A, B$, given by:
\begin{equation*}
    A =
    \begin{bmatrix}
    3 & -\infty\\
    7 & -\infty \\
    -\infty & 0\\
    \end{bmatrix} 
    \text{  and  }
    B =
   \begin{bmatrix}
    2 & -\infty \\
    -\infty & 1 \\
    -3 & 4 \\
    \end{bmatrix}.
\end{equation*}
The two-sided system $A\otimes x\leq B\otimes x$ is given by the following system of inequalities:
\begin{equation}
\label{e:exsystem}
\begin{split}
    3 + x_1 &\leq 2 + x_1\\
    7 + x_1 &\leq 1 + x_2\\
    x_2 &\leq \text{max}(-3 + x_1, 4 + x_2).\\
\end{split}
\end{equation}
We immediately see from the first inequality that there is no solution with $x_1>-\infty$. 
In fact, the solution set of this system is 
$\{(-\infty,t)\colon t\in\mathbb{R}_{\max}\}$.

The corresponding game is given on Figure~\ref{f:MPGexample} (left).
Here squares denote the nodes at which Max makes a move and circles denote the nodes at which Min makes a move.

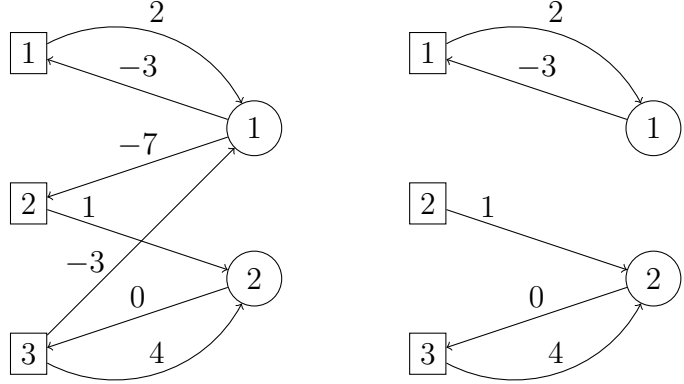
\begin{figure}
     \begin{center}
     \begin{tabular}{ccccc}
     \begin{tikzpicture}
    \node[shape=rectangle,draw=black] (A1) at (0,0) {1};
    \node[shape=rectangle,draw=black] (A2) at (0,-2) {2};
    \node[shape=rectangle,draw=black] (A3) at (0,-4) {3};
    \node[shape=circle,draw=black] (B1) at (3,-1) {1};
    \node[shape=circle,draw=black] (B2) at (3, -3) {2};

    \path [->](A1) edge[bend left = 45] node[above] {$2$} (B1);
    \path [->](B1) edge node[above] {$-3$} (A1);
    \path [->](B1) edge node[above] {$-7$} (A2);
    \path [->](A2) edge node[above] {} (B2); 
    \path [->](B2) edge node[above] {$0$} (A3); 
    \path [->](A3) edge[bend right = 45] node[above] {$4$} (B2);
    \path [->](A3) edge node[left, above] {} (B1);
 \node[] (x) at (0.8,-2.05) {$1$};
 \node[] (y) at (0.75,-2.8) {$-3$};
\end{tikzpicture}
&&&&
\begin{tikzpicture}
    \node[shape=rectangle,draw=black] (A1) at (0,0) {1};
    \node[shape=rectangle,draw=black] (A2) at (0,-2) {2};
    \node[shape=rectangle,draw=black] (A3) at (0,-4) {3};
    \node[shape=circle,draw=black] (B1) at (3,-1) {1};
    \node[shape=circle,draw=black] (B2) at (3, -3) {2};

    \path [->](A1) edge[bend left = 45] node[above] {$2$} (B1);
    \path [->](B1) edge node[above] {$-3$} (A1);
    \path [->](A2) edge node[above] {} (B2); 
    \path [->](B2) edge node[above] {$0$} (A3); 
    \path [->](A3) edge[bend right = 45] node[above] {$4$} (B2);
 \node[] (x) at (0.8,-2.05) {$1$};
\end{tikzpicture}
\end{tabular}
\caption{The mean payoff game corresponding 
to system \eqref{e:exsystem} and its restriction}
\label{f:MPGexample}
\end{center}
\end{figure}

Let players Max and Min choose positional strategies $\sigma\colon 1\to 1,\; 2\to 2,\; 3\to 2$
and $\tau\colon 1\to 1,\; 2\to 3$, respectively. The game is then played on a subgraph of the graph of the 
mean-payoff game. This subgraph is shown on Figure~\ref{f:MPGexample} (right). 

Clearly, 
the rewards of Max for the trajectories starting at the two nodes of Min are:
\begin{equation*}
\Phi_{A,B}(1,\sigma,\tau)=-1,\qquad \Phi_{A,B}(2,\sigma,\tau)=4.
\end{equation*}
Also, it can be checked that $\sigma^*=\sigma$ and $\tau^*=\tau$ is an equilibrium pair of 
strategies of Max and Min, in the sense of the saddle point property~\eqref{ValueMPG}.
The above rewards are the values of the games starting at nodes 1 and 2 of Min. Here, Max is winning if the game starts at node 2 of Min but losing 
if the game starts at node 1 of Min. This correlates with the fact that system~\eqref{e:exsystem}
has a solution with $x_2>-\infty$, but not with $x_1>-\infty$, as predicted by Proposition~\ref{p:tropicalMPG}.
\end{example}

In what follows, we will often abbreviate mean payoff game(s) to MPG.

\section{MPG representation of pseudolinear optimization}

\label{s:MPG}

\subsection{Pseudolinear optimization over alcoved polyhedra}


Consider the following problem.

\begin{problem}[Pseudolinear optimization over alcoved polyhedra]
Given $U\in\Rmax^{n\times n}$, $p, g\in\Rmax^n$, $q,h\in(\R\cup\{+\infty\})^n$ such that $\rho(U)\leq 0$ and $U\otimes g\leq h$, find  
\begin{equation*}
    \begin{aligned}
& \underset{x\in\R^n}{\text{min}}
& & x^- \otimes p \oplus q^- \otimes x\\
& \text{subject to}
& & g \leq x \leq h\\
&&& U\otimes x \leq x.
    \end{aligned}
\end{equation*}
and describe all $x\in\R^n$ for which the minimum is attained.
\label{prob:constrpseudolin}
\end{problem}

Note that whenever we have $q_i=+\infty$ it means that $q_i^-=-\infty$, so that $x_i$ does not appear in the objective function (but $x_i^-$ still may appear when $p_i\neq -\infty$). When $h_i=+\infty,$ it means that $x_i$ is not bounded from above by any constant, and if $g_i=-\infty$ then $x_i$ is not bounded from below.

Problem~\ref{prob:constrpseudolin} was considered and solved by Krivulin~\cite{KrivPseudolin-14}, for the case of real $q$ and $h$. The solution set to the system of constraints in this problem is an alcoved polyhedron. The geometry of such polyhedra is combining tropical and ordinary convexity: see, e.g., Joswig and Kulas~\cite{JK-10}, or  De La Puente and Claveria~\cite{PuenteCL-18} for a more recent reference.

It follows from the results of~\cite{KrivPseudolin-14} that an alcoved polyhedron described by
\begin{equation*}
P=\{x\colon g\leq x\leq h,\ U\otimes x\leq x\}    
\end{equation*}
is non-empty if and only if the conditions $\rho(U)\leq 0$ and $U^*\otimes g\leq h$ hold.
Note that these are precisely the conditions that are assumed in Problem~\ref{prob:constrpseudolin}. 

The proof of the proposition below is due to Krivulin~\cite{KrivPseudolin-14}, although the case where some entries of $h$ and $q$ are equal to $+\infty$ was not considered in that work.
See also Appendix of the present paper for a new alternative proof, which makes use of the connection to mean-payoff games.

\begin{proposition} [\cite{KrivPseudolin-14} Theorem 6]
\label{p:constrpseudolin}
The optimal value of Problem~\ref{prob:constrpseudolin} is 
\begin{equation*}
    \theta = (q^-\otimes U^*\otimes p)^{\otimes \frac{1}{2}} \oplus h^-\otimes U^*\otimes p \oplus q^-\otimes U^*\otimes g, 
\end{equation*}
If $\theta$ is finite then the solution set of this problem is
\begin{equation}
\label{e:solset-pseudolin}
    \{ U^*\otimes v : g \oplus \theta^-\otimes p \leq v \leq (U^*)^{\#} \otimes' (\theta\otimes q \oplus' h),\ v\in\R^n\}.
\end{equation}
\end{proposition}
Here and below, $t^{\otimes 1/2}$ is the same as $t/2$ for any $t\in\mathbb{R}_{\max}$ in the usual notation. (Note that one can define $t^{\otimes\alpha}:=\alpha\times t$ for arbitrary $t\in\mathbb{R}_{\max}$ and $\alpha\geq 0$, but we will not need it in this paper.)

\subsection{Pseudolinear optimization as parametric MPG}
\label{ss:recast}

The purpose of this section is to recast Problem~\ref{prob:mainpseudolin} as a parametric mean-payoff game.
To begin, we can rewrite that problem in the following way, introducing new variable $\lambda$:

\begin{equation}
\label{MP33}
    \begin{aligned}
& \underset{x\in\R^n,\lambda\in\R}{\min}
& & \lambda \\
& \text{subject to}
& &  x^-\otimes p\oplus q^-\otimes  x \leq \lambda, \\
&&& U\otimes x \oplus b \leq V\otimes x \oplus d.
    \end{aligned}
\end{equation}

The first inequality is equivalent to $x^-\otimes p\leq\lambda$ and $q^-\otimes x\leq\lambda$. The first of these inequalities is equivalent to $x_i^-\otimes p_i\leq\lambda$ for all $i$, which is the same as $p_i\leq x_i\otimes \lambda$ for all $i$. Hence we obtain that~\eqref{MP33} is equivalent to:
\begin{equation}
\label{MP3}
    \begin{aligned}
& \underset{x\in\R^n,\lambda\in\R}{\min}
& & \lambda \\
& \text{subject to}
& &  p \leq \lambda\otimes  x,\quad q^-\otimes  x \leq \lambda, \\
&&& U\otimes x \oplus b \leq V\otimes x \oplus d.
    \end{aligned}
\end{equation}

Introducing $ z = (y\ t)^T$ where $y\in\Rmax^n$ and $t\in\Rmax$, we see that a finite solution to the system of constraints in \eqref{MP3} exists if and only if the following parametric two-sided system

\begin{equation*}
A\otimes z\leq B(\lambda)\otimes z,
\end{equation*}
where
\begin{equation}
\label{e:ABlambda-pseudolin}
A =
    \begin{pmatrix}
    U & b\\
    -\infty & p\\
    q^- & -\infty \\
    \end{pmatrix} 
    \text{  and  }
    B(\lambda) =
   \begin{pmatrix}
    V & d\\
    \lambda\otimes I & -\infty \\
    -\infty & \lambda\\
    \end{pmatrix},
\end{equation}
has a finite solution ($z\in\R^{n+1}$).

Thus, we reformulate Problem~\ref{prob:mainpseudolin} as 
\begin{equation*}
\min\{\lambda\in\R \colon A\otimes z\leq B(\lambda)\otimes z\ \text{is solvable with $z\in\R^{n+1}$}\},    
    \end{equation*}
where $A$ and $B(\lambda)$ are given by~\eqref{e:ABlambda-pseudolin}.

 We now reformulate this problem in terms of MPG, using Proposition~\ref{p:tropicalMPG}, by introducing the function $\Phi(\lambda)$ defined as the minimal value of the MPG associated with the system $A\otimes z\leq B(\lambda)\otimes z$: 
\begin{equation*}
\min\{\lambda\in\R\colon\Phi(\lambda)\geq 0\}, 
\quad\text{where}\quad \Phi(\lambda) = \underset{i}{\text{ min }} \chi_i (A^{\#} B(\lambda)).
\end{equation*}

We also define functions $\Phi_{\tau}(\lambda)$ and $\Phi^{\sigma}(\lambda)$ corresponding to MPG where the strategies of Min and Max are restricted to $\tau$ and $\sigma$ respectively:

\begin{equation*}
    \Phi_{\tau}{(\lambda)} = \underset{i}{\text{ min }} \chi_i (A_{\tau}^{\#} B(\lambda)),\qquad
     \Phi^{\sigma}{(\lambda)} = \underset{i}{\text{ min }} \chi_i (A^{\#} B^{\sigma}(\lambda))
\end{equation*}
where
\begin{equation*}
(A_{\tau})_{ij}=
\begin{cases}
a_{ij} & \text{if $i=\tau(j)$}\\
-\infty &\text{otherwise}.
\end{cases},\qquad 
(B^{\sigma})_{ij}=
\begin{cases}
b_{ij} & \text{if $j=\sigma(i)$}\\
-\infty & \text{otherwise}.
\end{cases}
\end{equation*}

Proposition~\ref{p:MPG-value} then implies 
the following result:

\begin{proposition}
\label{p:Newt}
Let $\Sigma$ be the set of all strategies of player Min and $T$ is the set of all strategies of player Max. Then
    \begin{equation}
    \label{PhiEquals}
    \underset{\tau \in T}{\min} \; \Phi_{\tau}(\lambda) \; = \; \Phi(\lambda) \; = \; \underset{\sigma \in \Sigma}{\max}\; \Phi^{\sigma}(\lambda).
\end{equation}
\end{proposition}

The following elementary properties of 
$\Phi(\lambda),$ $\Phi^{\sigma}(\lambda)$ and
$\Phi_{\tau}(\lambda)$ are the same as in~\cite{BiNewt}[Theorem 8].

\begin{proposition}
\label{p:nondecrease}
For $A$ and $B(\lambda)$ given ~\eqref{e:ABlambda-pseudolin}, functions
$\Phi(\lambda)$, $\Phi^{\sigma}(\lambda)$ and 
$\Phi_{\tau}(\lambda)$ are non-decreasing, piecewise-linear and continuous. 
\end{proposition}
\if{
\begin{proof}
Let us introduce the following functions
\begin{equation}
\label{e:flambda}
    f(\lambda,x) = A^{\#}B(\lambda)x, \quad 
    f^{\sigma}(\lambda,x) = A^{\#}B^{\sigma}(\lambda)x, \quad f_{\tau}(\lambda,x) = A_{\tau}^{\#}B(\lambda)x, \quad 
\end{equation}
and observe that they are non-decreasing and continuous functions of $\lambda$ for any fixed $X\in\Rmax^n$, since the finite entries of $B(\lambda)$ are non-decreasing and continuous functions of $\lambda$ (and by the properties of max-plus arithmetics). The claim the follows from the properties of limits, since  $\chi(f)=\lim_{k\to\infty} f^k(0)/k$ by the definition of the cycle-time vector and since the functions $\Phi$ are defined as minimal components of such vectors.
\end{proof}
}\fi

\subsection{MPG diagram of the problem}

Analysing $A$ and $B(\lambda)$ of~\eqref{e:ABlambda-pseudolin} we can build an MPG diagram corresponding to this two-sided system, as shown on Figure~\ref{f:MPG-pseudolin}. On this diagram we can see three groups of square nodes of Max: $[m]$, $[n]$ and $[1]$, with groups of arcs coming in and out of them, corresponding to $U\otimes x\oplus b\leq V\otimes x\oplus d$, 
$p\leq\lambda\otimes x$ and $q^-\otimes x\leq \lambda$, respectively (where $m,n$ and $1$ indicate the numbers of nodes in each group). Min has just two groups of circle nodes: $[n]$ and $[1]$, corresponding to the variables and to the free-standing column, respectively.
An arc between two nodes of any two groups of nodes exists if and only if the corresponding entry of the matrix or the vector, by which the group connection is marked, is finite. 


\begin{figure}[h!]
\begin{center}
\begin{tikzpicture}[scale=1]
    \node[shape=rectangle,draw=black] (m) at (0,3) {[m]};
    \node[shape=circle,draw=black] (n) at (5,6) {[n]};
		\node[shape=rectangle,draw=black] (m+k+1) at (5,3) {[1]};
    \node[shape=circle,draw=black] (n+1) at (5,0) {[1]};
    \node[shape=rectangle,draw=black] (k) at (10,3) {[n]};
    
    \path [->](m) edge[bend left = 10] node[above] {$V$} (n);
    \path [->](n) edge[bend left = 10] node[below] {$U^{\#}$} (m);
    \path [->](m) edge[bend left = 10] node[above] {$d$} (n+1);
    \path [->](n+1) edge[bend left = 10] node[below] {$b^-$} (m);
    \path [->](k) edge node[above] {$\lambda I$} (n);
    \path [->](n+1) edge node[above] {$p^-$} (k);
    \path [->](m+k+1) edge node[right] {$\lambda$} (n+1);
    \path [->](n) edge node[right] {$q$} (m+k+1);
\end{tikzpicture}
\caption{MPG corresponding to tropical pseudolinear programming\label{f:MPG-pseudolin}}
\end{center}
\end{figure}
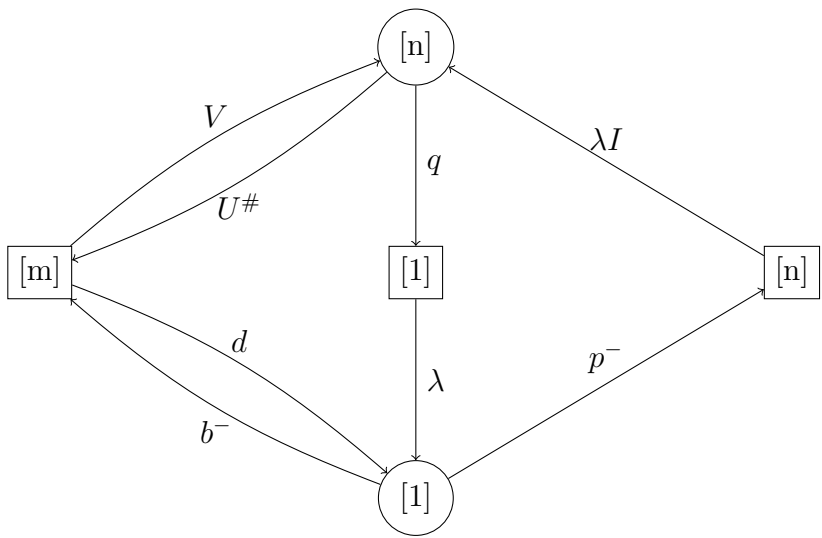

Using this MPG diagram we can establish the following optimality and unboundedness certificates for Problem~\ref{prob:mainpseudolin}, similar to Theorems 12 and 13 in~\cite{BiNewt}.
Both certificates refer to some groups of nodes of Max shown on Figure~\ref{f:MPG-pseudolin}. Their proofs, being similar to those in~\cite{BiNewt}, are deferred to Appendix.

\begin{proposition}[Optimality certificate]
\label{p:optcert}
$\lambda^*$ is optimal if and only if $\Phi(\lambda^*)\geq 0$ and there exists 
$\tau$ such that in the mean-payoff game defined by $A_{\tau}$ and $B(\lambda^*)$ all cycles accessible from some node of Min have non-positive
weights, and all cycles of zero weight accessible from that node of Min contain a node of Max that does not belong to the left group $[m]$ of its nodes.
\end{proposition}

\begin{proposition}[Unboundedness certificate]
\label{p:unbcert}
The problem is unbounded if and only if, for some $\sigma$, all cycles of the graph defined by $A$ and $B^{\sigma}(0)$ contain nodes of Max that are only from the $[m]$ group (on the left) and have a nonnegative weight.
\end{proposition}

We now consider the case when all data in the problems are integer.

\begin{proposition}
\label{p:intcase}
When the finite entries of $p,q,U,b,V,d$ are all integer, the optimal value of Problem~\ref{prob:mainpseudolin}, if finite,  
is an integer multiple of $1/2$.
\end{proposition}
\begin{proof}
 For the MPG on Figure~\ref{f:MPG-pseudolin}, for each $\lambda$ there is $\epsilon>0$ in which 
 $$
 \Phi(\mu)=\Phi_{A,B(\mu)}(j,\sigma,\tau)
 $$
 for some $j,$ $\tau$ and $\sigma$ and 
 all $\mu\in [\lambda-\epsilon,\lambda]$.
 This implies that 
 $$
 \lambda^*=\min\{\lambda\colon \Phi_{A,B(\lambda)}(j,\sigma,\tau)\geq 0\}=
 \min\{\lambda\colon k\lambda+s\geq 0\},
 $$
 for some $j,$ $\sigma$, $\tau$, $k$ and $s$. 
 Here we recall that $\Phi_{A,B(\lambda)}(j,\sigma,\tau)$ is the weight of a cycle (divided by the number of turns). Inspecting Figure~\ref{f:MPG-pseudolin} we can see that a cycle can collect no more than two  repetitions of $\lambda$, hence in the above expression for $\lambda^*$ we can have $k$ equal to $1$ or $2$, which means that the denominator of the optimal value is bounded by $2$.
\end{proof}

\section{Bisection method}
\label{s:bisection}

Since $\Phi(\lambda)$ is a non-decreasing function, we can use the bisection method to find $\min\{\lambda\colon\Phi(\lambda)\geq 0\}$. Recall that $\Phi(\lambda)\geq 0$
is equivalent to existence of $x\in\R^n$
that solves
\begin{equation}
     \label{e:solvability}
     \begin{split}
    & x^-\otimes p\oplus q^-\otimes x\leq\lambda,\\
    & U\otimes x\oplus b\leq V\otimes x\oplus d.
     \end{split}
 \end{equation}

In general, this method can be only approximate, but in the case of integer input (i.e., when all finite entries of $U$, $V,$ $b$, $d$, $p$, $q$ are integer) it can be made precise, since by Proposition~\ref{p:intcase} in this case the  optimal $\lambda$ is an integer multiple of $1/2$. In the description of the bisection method given below, rounding up ($\lceil\cdot\rceil$) means finding the least $\lambda$ greater than or equal to the given number and an integer multiple of $1/2$. Similarly, rounding down ($\lfloor\cdot\rfloor$) means finding the biggest $\lambda$ less than or equal to the given number and an integer multiple of $1/2$.   

Before we give a description of the bisection method, we need to have the upper and the lower bounds, with which we can start. 

{\bf Upper bound:} Using an MPG solver such as described in~\cite{DGShort-06} or~\cite{AltMethod}, find $x\in\R^n$ such that $U\otimes x\oplus b\leq V\otimes x\oplus d$. Then compute 
\begin{equation}
\label{e:ubound}
\lambda_1^{(+)}= \left\lfloor x^-\otimes p\oplus q^-\otimes x\right\rfloor
\end{equation}
for this $x$. Then we know that $\Phi(\lambda^{(+)}_1)\geq 0$.

{\bf Lower bound:}
We define 
$\lambda_1^{(-)}$ as 
the optimal value of the unconstrained problem $$\min\limits_{x\in\R^n}  x^-\otimes p\oplus q^-\otimes x.$$ Using the result of Krivulin~\cite{KrivPseudolin-12}[Theorem 4] (a slight extension to the case $q\in(\R\cup\{+\infty\})^n$), we have 
\begin{equation}
\label{e:lbound}
\lambda_1^{(-)}=(q^-\otimes p)^{\otimes \frac{1}{2}}
\end{equation}
As this unconstrained problem is a relaxation of the problem in question, we either have that $\Phi(\lambda_1^{(-)})\geq 0$ and then $\lambda_1^{(-)}$ is the optimal value of the constrained problem, or we have $\Phi(\lambda_1^{(-)})<0$ and then we continue with the bisection method. If $q^-\otimes p$ is finite, which we assume, then the problem is bounded from below. Note that $q^-\otimes p$ is finite if and only if there exists at least one $i$ such that both $q_i$ and $p_i$ are finite. We now describe the bisection algorithm.

 \begin{algorithm}[h]
 \renewcommand{\algorithmicrequire}{\textbf{Input:}}
 \renewcommand{\algorithmicensure}{\textbf{Output:}}
 \caption{Integer Bisection \label{a:bisection}}
 \begin{algorithmic}[1] 
 \REQUIRE  $U$, $V$, $b$, $d$, $p$ with integer or $-\infty$ entries, $q$ with integer or $+\infty$ entries.
 \STATE Compute $\lambda_1^{(-)}$ by~\eqref{e:lbound}. 
 \IF{$\Phi(\lambda_1^{(-)})\geq 0$} 
\STATE Find a finite solution $x$ to~\eqref{e:solvability} with $\lambda=\lambda_1^{(-)}$ and return.
 \ELSE 
 \STATE Compute $\lambda_1^{(+)}$ by~\eqref{e:ubound} and proceed with $k=1$.
 \ENDIF
 \WHILE{$\lambda_k^{(-)}<\lambda_k^{(+)}$}
 \STATE $\lambda_k=(\lambda_k^{(-)}+\lambda_k^{(+)})/2$.
 \IF{$\Phi(\lambda_k)<0$}
 \STATE $\lambda_{k+1}^{(+)} = \lambda_{k}^{(+)}$ and $\lambda_{k+1}^{(-)} = \lceil\lambda_{k}\rceil$.
 \ELSE 
 \STATE ${\lambda_{k+1}^{(+)}} = \lfloor\lambda_{k}\rfloor$ and $\lambda_{k+1}^{(-)} = \lambda_{k}^{(-)}$.
 \ENDIF
 \STATE $k=k+1$.
 \ENDWHILE
 \STATE Find
 a finite solution $x$ to~\eqref{e:solvability} with $\lambda=\lambda_k^{(+)}=\lambda_k^{(-)}$ and return.
\ENSURE $(\lambda,x)$.
 \end{algorithmic}
 \end{algorithm}

It is clear that in this algorithm, for each $\lambda_k^{(+)}$ we have $\Phi(\lambda_k^{(+)})\geq 0$, and for each 
$\lambda_k^{(-)}$ we have $\Phi(\lambda_k^{(-)})\leq 0$ and $\Phi(\lambda')<0$ whenever $\lambda'<\lambda_k^{(-)}$. This implies that when $\lambda=\lambda_k^{(+)}=\lambda_k^{(-)}$, we 
indeed have $\Phi(\lambda)=0$ and $\Phi(\lambda')<0$ for any 
$\lambda'<\lambda$, so this $\lambda$ is the optimal value of the problem. 

\begin{example} 
\label{ex:badex}
{\rm This example demonstrates that $(q^-\otimes p)^{\otimes 1/2}$ does not always work as a good lower bound. Consider a pseudolinear optimization problem (Problem~\ref{prob:mainpseudolin}) with
\begin{equation*}
\begin{split}
U&=\left(
      \begin{array}{cc}
        0 & -\infty \\
       -\infty & 0\\
      \end{array}
    \right),\quad
b=\left(
    \begin{array}{c}
      -\infty \\
      -\infty \\
    \end{array}
  \right),\quad
V=\left(
    \begin{array}{cc}
      -\infty & 0 \\
      0 & -\infty \\
    \end{array}
  \right),\quad 
d=\left(
    \begin{array}{c}
      -\infty \\
      -\infty \\
    \end{array}
  \right),\\
p&=\left(
    \begin{array}{c}
      0 \\
      -\infty \\
    \end{array}
  \right),\qquad
q=\left(
    \begin{array}{c}
      +\infty \\
      0 \\
    \end{array}
  \right).
\end{split}
\end{equation*}

In this example, $p_{i}$ and $q_{i}$ are not finite at the same time for any $i$, so $(q^{-}\otimes p)^{\otimes 1/2}=-\infty$, but the problem is bounded and its optimal value is $0$. Indeed, we are seeking minimum over $\max(x_1^{-},x_2)$ on the line $x_1=x_2$, which is equal to $0$.} 
\end{example}

\section{Newton Iterations}
\label{s:Newt}

\subsection{Formulation}

Let us begin this section by introducing the concepts of left-optimal strategy.
By the right-hand side of~\eqref{PhiEquals}, $\Phi(\cdot)$ is a pointwise maximum of a finite number of functions $\Phi^{\sigma}(\cdot)$, for $\sigma\in\Sigma$.  Therefore, for each $\lambda$ there exists $\Sigma'\subseteq\Sigma$ such that $\Phi^{\sigma}(\lambda)=\Phi(\lambda)$ for each $\sigma\in\Sigma'$. Furthermore, each of these $\Phi^{\sigma}(\cdot)$ is a piecewise-linear and continuous function, so there exists a small enough $\epsilon>0$ such that each $\Phi^{\sigma}(\cdot)$ with $\sigma\in\Sigma'$ is linear on 
$[\lambda-\epsilon,\lambda]$ and is bigger than any $\Phi^{\sigma}(\cdot)$ with $\sigma\notin\Sigma'$. This implies that there exists $\sigma^*$ such that $\Phi^{\sigma^*}(\mu)=\Phi(\mu)$ for all $\mu\in[\lambda-\epsilon,\lambda]$. Such $\sigma^*$ is called a {\em left-optimal strategy} at $\lambda$. Left-optimal strategies play the role of derivatives in Algorithm~\ref{a:Newt} stated below. 
Iterations of this algorithm will refer to efficient solution of the following problem
\begin{equation}
\label{e:subproblem}
    \min \{ \lambda : \Phi^{\sigma}(\lambda) \geq 0\},
\end{equation}
which will be explained in Subsection~\ref{sss:partial}.

\begin{algorithm}[h]
\renewcommand{\algorithmicrequire}{\textbf{Input:}}
 \renewcommand{\algorithmicensure}{\textbf{Output:}}
 \caption{Newton iterations, left-optimal strategies\label{a:Newt}}
 \begin{algorithmic}[1] 
\REQUIRE $U,$ $V,$ $b,$ $d$, $p$ with entries in 
$\Rmax$ and $q$ with entries in $\R\cup\{+\infty\}$.
\STATE Set $\lambda_0=+\infty$, compute $\lambda_1$ by~\eqref{e:ubound} and proceed with $k=1$.
\WHILE{$\lambda_k<\lambda_{k-1}$}
\STATE Find a left-optimal strategy $\sigma_k$ of Player Max at $\lambda_k$. 
\STATE Solve~\eqref{e:subproblem} with $\sigma=\sigma_k$
and let $\lambda_{k+1}$ be the optimal value of~\eqref{e:subproblem}.
\STATE $k=k+1$.
\ENDWHILE
\STATE Find a solution $x\in\R^n$ to system~\eqref{e:solvability} with $\lambda=\lambda_k$. 
\ENSURE $(\lambda,x)$.
\end{algorithmic}
\end{algorithm}

 The following result and its proof are similar to the corresponding statement from Gaubert et al.~\cite{BiNewt}, therefore the proof is omitted.

\begin{proposition}
Newton algorithm  finds a solution of a tropical pseudoquadratic optimization problem in a finite number of steps, limited by the number of strategies of player Max in the associated MPG.
\end{proposition}
\if{
\begin{proof}
Before the stopping condition is met, we find $\lambda_{k+1}=\min\{\lambda\colon\Phi^{\sigma}(\lambda)\geq 0\}$, where $\sigma$ is a left-optimal or left-winning strategy at $\lambda_k$. This implies that $\Phi^{\sigma}(\lambda)<0$ for all $\lambda<\lambda_{k+1}$. Therefore $\sigma$ cannot be 
left-optimal for any $\lambda< \lambda_{k+1}$ where $\Phi(\lambda)\geq 0$, and 
it can be left-optimal at $\lambda_{k+1}$ only if the algorithm stops there. As player Max has a finite number of strategies and they can never repeat until the end, it follows that the algorithm
stops after a finite number of steps.

When the algorithm 
stops at $\lambda^*$, we know that the left-optimal strategy $\sigma$ has a property that 
$\Phi^{\sigma}(\lambda)<0$ for $\lambda<\lambda^*$.
Since $\sigma$ is left-optimal, we have $\Phi(\lambda)=\Phi^{\sigma}(\lambda)<0$ for $\lambda\in [\lambda^*-\epsilon,\lambda^*]$ for some $\epsilon$, and also $\Phi(\lambda)<0$ for 
all $\lambda<\lambda^*$ since $\Phi(\lambda)$ is non-decreasing. We also have $\Phi(\lambda^*)=\Phi^{\sigma}(\lambda)$, showing that $\lambda^*$ is an optimal value.

\end{proof}
}\fi

In the general case (i.e., when the data are arbitrary real numbers) left-optimal  strategies can be found using the algebra of germs, for which we also refer the reader to Gaubert et al.~\cite{BiNewt}.


\subsection{The case of integer data}

In this section we will discuss how to implement Newton's algorithm in the case when all given data (i.e., coefficients of $U$, $V$, $b$, $d$, $p$ and $q$) are integer.

We first explain how to find the left-optimal strategies in the case of integer data. By the arguments similar to those of Proposition~\ref{p:intcase}, $\lambda_k$ appearing in Newton iterations are of the form $l/2$ where $l$ is an integer, while the breakpoints of $\Phi(\lambda)$ can be at rational points with denominators not exceeding $2(n+1)$.
The greatest denominator of a distance between $\lambda_k$ and such a breakpoint is bounded from above by $4(n+1)$. 
Therefore, if we take  $\epsilon=1/4(n+1)$, we can be sure that if $\sigma_k$ is optimal at $\lambda_k-\epsilon$ then it is optimal for any 
$\lambda\in [\lambda_k-\epsilon,\lambda_k]$, thus left-optimal. 

However, in the case of integer data, instead of finding a left-optimal strategy we can find an optimal strategy $\sigma_k$ at $\lambda_k^{[-]}$, defined as the largest multiple of $1/2$ less than $\lambda_k$ if it is not a multiple of $1/2$ (which may happen if $k=1$), and as $\lambda_k-1/2$ if it is. We use the notation $\lambda_k^{[-]}$ here to distinguish it from the notation $\lambda_k^{(-)}$
(and $\lambda_k^{(+)}$) used in the bisection method.
 
 If we have $\Phi^{\sigma_k}(\lambda_k^{[-]})\geq 0$, then the Newton iterations proceed, and if 
 $\Phi^{\sigma_k}(\lambda_k^{[-]})<0$ then $\Phi(\lambda)<0$ for all $\lambda<\lambda_k$ and we must stop. 
 So we have the following modification of Algorithm~\ref{a:Newt}, where optimal strategies can be found using some MPG solvers such as the algorithm by Dhingra and Gaubert~\cite{DGShort-06}.


\begin{algorithm}[h]
\renewcommand{\algorithmicrequire}{\textbf{Input:}}
 \renewcommand{\algorithmicensure}{\textbf{Output:}}
 \caption{Integer Newton iterations\label{a:Newt-int-opt}}
 \begin{algorithmic}[1] 
\REQUIRE $U,$ $V,$ $b,$ $d$, $p$ with integer or $-\infty$ entries, and $q$ with integer or $+\infty$ entries.
\STATE Set $\lambda_0=+\infty$, compute $\lambda_1$ by~\eqref{e:ubound} and proceed with $k=1$.
\WHILE{$\lambda_k<\lambda_{k-1}$}
\STATE Compute $\lambda_k^{[-]}$ and 
find an optimal strategy $\sigma_k$ of player Max at $\lambda_k^{[-]}$.
\STATE Solve~\eqref{e:subproblem} with $\sigma=\sigma_k$ and let $\lambda_{k+1}$ be the optimal value of~\eqref{e:subproblem}.
\STATE $k=k+1$.
\ENDWHILE
\STATE Find a finite solution $x$ to system~\eqref{e:solvability} with $\lambda=\lambda_{k-1}$. 
\ENSURE $(\lambda,x)$.
\end{algorithmic}
\end{algorithm}

Note that in this algorithm the sequence of $\lambda_i$ is strictly decreasing until the last step, at which we may have $\lambda_k=\lambda_{k-1}$, $\lambda_k>\lambda_{k-1}$ or even $\lambda_k=+\infty$ if the problem~\eqref{e:subproblem} is infeasible for $\sigma=\sigma_k$.


\subsection{Finding the least zero of a partial spectral function}
\label{sss:partial}

Here we discuss how to solve~\eqref{e:subproblem}.
The problem can be translated back to two-sided system where it becomes the problem of finding 
\begin{equation}
\label{e:subproblem1}    
\min\{\lambda\in\R\colon A\otimes z\leq B^{\sigma}(\lambda)\otimes z\quad \text{has a solution $z\in\R^{n+1}$}\}, \end{equation}
where $B^{\sigma}$ is the matrix with entries equal to $b^{\sigma}_{ij}=b_{ij}$ when $j=\sigma(i)$ and to $-\infty$ otherwise, and $\sigma$ is the strategy of Max appearing in~\eqref{e:subproblem}. Recalling~\eqref{e:ABlambda-pseudolin}, we can rewrite ~\eqref{e:subproblem1} as
\begin{equation}
\label{e:subproblem2}
\min\limits_{\lambda\in\R,\;x\in\R^n} \lambda\ \text{s.t.}\  p \leq \lambda\otimes x,\quad  q^-\otimes x \leq \lambda,\quad U\otimes  x \oplus b \leq V^{\sigma}\otimes  x \oplus d^{\sigma}.
\end{equation}
where $(V^{\sigma}\; d^{\sigma})=(V\; d)^{\sigma}$, and $(V\; d)^{\sigma}$ is defined as $B^{\sigma}$ above: we leave the entries of $B$ for which $j=\sigma(i)$ untouched and set all the remaining entries to $-\infty$.
We now work with the system of constraints 
\begin{equation}
\label{e:reduced}
U\otimes x\oplus b\leq V^{\sigma}\otimes x\oplus d^{\sigma}.
\end{equation}

\begin{proposition}
\label{p:sysreduced}
System~\eqref{e:reduced} 
is equivalent to
\begin{equation}
\label{e:reduced1}
    F\otimes x_J \oplus G\otimes x_{\overline{J}} \leq x_J, \quad  l_J \leq x_J \leq  u_J \quad x_{\overline{J}} \leq u_{\overline{J}}.
\end{equation}
and further to
\begin{equation*}
R\otimes x\leq x,\quad l\leq x\leq u\quad \text{for}\ 
R=
\begin{pmatrix}
F & G\\
-\infty & -\infty
\end{pmatrix}
\end{equation*}
for some set $J\subseteq [n]$, its complement $\overline{J}=[n]\setminus J$, matrices $F,G$, and vectors $l$ and $u$.
\end{proposition}

\begin{proof}
Define the index set $J$ as follows:
\begin{equation*}
    J = \{ j \in [n]: \sigma(i) = j \quad \text{for some $i$}\}.
\end{equation*}

Now for any $j\in J$ consider all $i$ such that $\sigma(i)=j$. 
Denote the set of such $i$ by $I_j$ and observe that 
$
\bigcup_{j\in J\cup \{n+1\}} I_j= [m].
$
We have two cases:

\textbf{Case 1:} $j\in J$. In this case we obtain 
\begin{equation}
\label{e:xjatom}
  \bigoplus_{k=1}^{n} (v_{ij}^{-}\otimes u_{ik} \otimes x_k)\oplus v_{ij}^{-}\otimes b_i \leq  x_j,\qquad j\in J,\ i\in I_j
\end{equation}
for all such $i\in I_j$, and summing it up over $i\in I_j$ we have
\begin{equation}
\label{e:xjineq}
\bigoplus_{k=1}^{n}\bigoplus_{i\in I_j} (v_{ij}^{-}\otimes u_{ik}\otimes  x_k)\oplus \bigoplus_{i\in I_j} (v_{ij}^{-}\otimes b_i) \leq  x_j,\qquad j\in J.
\end{equation}
Observe that~\eqref{e:xjineq} is equivalent to the system of inequalities~\eqref{e:xjatom} where $i$ runs over $I_j$, and equivalent to 
the subsystem of~\eqref{e:reduced}, consisting of inequalities $i$ such that $\sigma(i)\in [n]$. 

\textbf{Case 2.} $j=n+1$. In this case we obtain
\begin{equation*}
  \bigoplus_{k=1}^{n} (u_{ik}\otimes  x_k)\oplus b_i \leq  d_i,\qquad i\in I_{n+1}
\end{equation*}
for all $i\in I_{n+1}$. Considering these inequalities for all $i\in I_{n+1}$ and seeing that $b_i\leq d_i$ in this case is just a necessary condition for~\eqref{e:reduced} to be consistent, we
see that the system of such inequalities taken over $i\in I_{n+1}$
is equivalent to the system
\begin{equation}
\label{e:upperbounds}
x_k\leq \min_{i\in I_{n+1}} u_{ik}^{-}\otimes 'd_i. \end{equation}
This is a system of upper bounds on $x_k$ (some of them can be 
equal to $+\infty$ if all the corresponding $u_{ik}$ are $-\infty$). 

Equation~\eqref{e:reduced} is thus equivalent to the system of~\eqref{e:xjineq} and~\eqref{e:upperbounds}. Combining these two, one can easily recognize~\eqref{e:reduced1}.

Next we observe that the constraint given by $F\otimes x_I \oplus G\otimes x_J \leq x_I$ can be written as
\begin{equation*}
    \begin{bmatrix}
    F & G \\
    -\infty & -\infty \\
    \end{bmatrix}\otimes 
    \begin{bmatrix}
    x_I\\
    x_J\\
    \end{bmatrix}
    \leq 
     \begin{bmatrix}
    x_I\\
    x_J\\
    \end{bmatrix} \\
\end{equation*}
which is equivalent to $R\otimes x\leq x$.
\end{proof}

The above proposition implies that the problem which we have to 
solve at each iteration is the following problem
\begin{equation}
\label{e:alcovedpseudolin}
\begin{split}
 \min_{x\in\R^n} x^-\otimes p \oplus q^-\otimes x \quad 
 \text{s.t.}\quad  l \leq x \leq u\ \text{and}\  
R\otimes x \leq x.
\end{split}
\end{equation}
Recognizing Problem~\ref{prob:constrpseudolin}, we recall that Proposition~\ref{p:constrpseudolin} yields the optimal value 
\begin{equation}
\label{e:thetalambdak}
    \theta = (q^-\otimes R^*\otimes p)^{\otimes \frac{1}{2}} \oplus u^-\otimes R^*\otimes p \oplus q^-\otimes R^*\otimes l
\end{equation}
and the solution set
\begin{equation}
\label{e:x}
    x = \{ R^*\otimes  v : l \oplus \theta^{-}\otimes p \leq v \leq (R^*)^{\#}\otimes'(\theta\otimes q\oplus' u),\ v\in\R^n \}.
\end{equation}
of this problem.
Thus we can compute $\lambda_k=\theta$ using~\eqref{e:thetalambdak} and, at optimality we can take any vector from~\eqref{e:x} as a solution of the system of constraints in~\eqref{e:subproblem2} and hence also system~\eqref{e:solvability} with 
$\lambda=\lambda_k=\theta$. For a finite vector $v$ we can take a vector with the following components:
\begin{equation*}
v_i=
\begin{cases}
((R^*)^{\#}\otimes'(\theta\otimes q\oplus' u))_i, & \text{if $((R^*)^{\#}\otimes'(\theta\otimes q\oplus' u))_i\neq +\infty$}, \\
l_i \oplus \theta^{-}\otimes p_i, & \text{otherwise if $l_i \oplus \theta^{-}\otimes p_i\neq -\infty$},\\
0, & \text{otherwise.}
\end{cases}
\end{equation*}
Note that the computation of $\theta$ in ~\eqref{e:thetalambdak} requires no more than $O(n^3)$ operations and can be performed very efficiently by shortest path algorithms.

\section{Example and numerical experiments}
\label{s:appl}

\subsection{Example}
\label{ss:example}

Consider the pseudolinear optimization problem where
\begin{equation}
\label{e:example}
\begin{split}
& U=\left(
      \begin{array}{cc}
        -\infty & -2 \\
        3 & -\infty \\
      \end{array}
    \right),\;
b=\left(
    \begin{array}{c}
      -\infty \\
      -\infty \\
    \end{array}
  \right),\;
V=\left(
    \begin{array}{cc}
      1 & 0 \\
      -\infty & 1 \\
    \end{array}
  \right),\;
d=\left(
    \begin{array}{c}
      -\infty \\
      1 \\
    \end{array}
  \right),\\
& p=\left(
    \begin{array}{c}
      0 \\
      -\infty \\
    \end{array}
  \right),\quad
q=\left(
    \begin{array}{c}
      -1 \\
      0 \\
    \end{array}
  \right).
\end{split}
\end{equation}

\subsubsection{Application of bisection (Algorithm~\ref{a:bisection})}

\textbf{Start:} We first compute $\lambda_{1}^{(-)}=(q^{-}\otimes p)^{\otimes\frac{1}{2}}=\frac{1}{2}$. We find that $\Phi(\lambda_{1}^{(-)})=-\frac{1}{3}<0$, so we proceed. 
Using the alternating method of~\cite{AltMethod}, we obtain a finite solution $x^{0}=(-8\  -8)^{\top}$ for $U\otimes x\oplus b\leq V\otimes x\oplus d$. Let $\lambda_{1}^{(+)}=(x^{0})^{-}\otimes p\oplus q^{-}\otimes x^{0}=8$. 

\textbf{Iterations:}
For $k=1$, let $\lambda_{1}=\frac{\lambda_{1}^{(+)}+\lambda_{1}^{(-)}}{2}=\frac{17}{4}$. Since $\Phi(\lambda_{1})=2>0$, set $\lambda_{2}^{(+)}=\lfloor\lambda_{1}\rfloor=4$ and $\lambda_{2}^{(-)}=\lambda_{1}^{(-)}=\frac{1}{2}$. 

For $k=2$, we check that $\lambda_2^{(+)}\neq\lambda_2^{(-)}$ and compute $\lambda_{2}=\frac{\lambda_{2}^{(+)}+\lambda_{2}^{(-)}}{2}=\frac{9}{4}$. Since $\Phi(\lambda_{2})=\frac{5}{6}>0$, set $\lambda_{3}^{(+)}=\lfloor\lambda_{2}\rfloor=2$ and $\lambda_{3}^{(-)}=\lambda_{2}^{(-)}=\frac{1}{2}$.

For $k=3$, we check that $\lambda_3^{(+)}\neq \lambda_3^{(-)}$ and compute
$\lambda_{3}=\frac{\lambda_{3}^{(+)}+\lambda_{3}^{(-)}}{2}=\frac{5}{4}$. Since $\Phi(\lambda_{3})=\frac{1}{6}>0$, set $\lambda_{4}^{(+)}=\lfloor\lambda_{3}\rfloor=1$ and $\lambda_{4}^{(-)}=\lambda_{3}^{(-)}=\frac{1}{2}$.

For $k=4$,  we check that $\lambda_4^{(+)}\neq \lambda_4^{(-)}$ and compute $\lambda_{4}=\frac{\lambda_{4}^{(+)}+\lambda_{4}^{(-)}}{2}=\frac{3}{4}$. Since $\Phi(\lambda_{4})=-\frac{1}{6}<0$, set $\lambda_{5}^{(+)}=\lambda_{4}^{(+)}=1$ and $\lambda_{5}^{(-)}=\lceil\lambda_{4}\rceil=1$.

We have $\lambda_5^{(-)}=\lambda_5^{(+)}=1$, so we stop, this is the optimal value of the problem. Using the alternating method, we obtain a finite solution $x=(-1\ 1)^{T}$ for system~\eqref{e:solvability} with $\lambda=\lambda_{5}^{(-)}=1$. Then return $(\lambda_{5}^{(-)}, x)$ as a solution of the problem.

\subsubsection{Application of Newton method (Algorithm~\ref{a:Newt-int-opt})}

\textbf{Start.} We begin with the same $\lambda_1=8$ as 
in the bisection method.

{\bf Iterations:} We first find an optimal strategy $\sigma_1$ at $\lambda_1^{[-]}=7.5$ with $\sigma_1(1)=1$, $\sigma_1(2)=3$, 
where $1$ and $2$ are nodes of Max that correspond to the inequalitites of the system $U\otimes x\oplus b\leq V\otimes x\oplus d$ (recall that Max does not have choice at any other node of the associated MPG).

Following Proposition~\ref{p:sysreduced}, we find that
$\min\{\lambda: \Phi^{\sigma_1}(\lambda)\geq 0\}$ is equivalent to~\eqref{e:alcovedpseudolin}, where
$$R=\left(
  \begin{array}{cc}
    -\infty & -3 \\
    -\infty & -\infty \\
  \end{array}
    \right),\quad
l=\left(
    \begin{array}{c}
      -\infty \\
      -\infty \\
    \end{array}
    \right),\quad
u=\left(
    \begin{array}{c}
      -2 \\
      +\infty \\
    \end{array}
  \right).$$
Then we obtain $$\lambda_{2}=(q^{-}\otimes R^{\ast}\otimes p)^{\otimes \frac{1}{2}}\oplus u^{-}\otimes R^{\ast}\otimes p\oplus q^{-}\otimes R^{\ast}\otimes l=2.$$ 

At iteration 2 we check that $\lambda_{2}<\lambda_{1}$ and find an optimal  strategy $\sigma_2$ at $\lambda_{2}^{[-]}=1.5$ with $\sigma_2(1)=2$ and $\sigma_2(2)=2$. 
Following Proposition~\ref{p:sysreduced}, we next solve~\eqref{e:alcovedpseudolin} with 
$$R=\left(
  \begin{array}{cc}
    -\infty & -\infty \\
    2 & -\infty \\
  \end{array}
    \right),\quad
l=\left(
    \begin{array}{c}
      -\infty \\
      -\infty \\
    \end{array}
    \right),\quad
u=\left(
    \begin{array}{c}
      +\infty \\
      +\infty \\
    \end{array}
  \right).$$
We obtain 
$$\lambda_{3}=(q^{-}\otimes R^{\ast}\otimes p)^{\otimes \frac{1}{2}}\oplus u^{-}\otimes R^{\ast}\otimes p\oplus q^{-}\otimes R^{\ast}\otimes l=1.$$

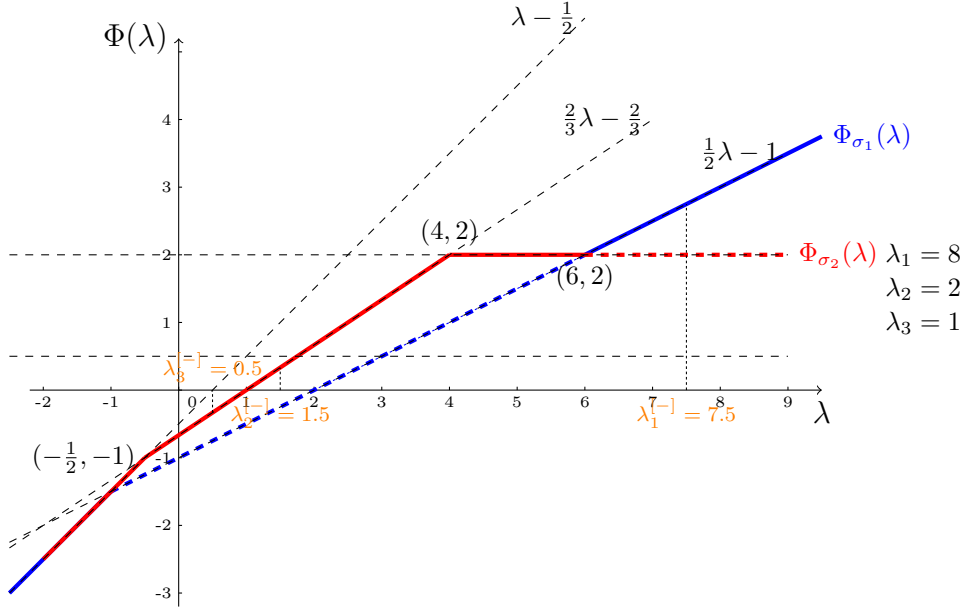
\begin{figure}
\begin{center}
\begin{tikzpicture}[scale=0.9]
\draw[->] (-2.2,0)--(9.5,0) node[below] {$\lambda$};
\draw[->] (0,-3.2)--(0,5.2) node[left] {$\Phi(\lambda)$};

\foreach \x in {0,1,...,11}
{
    \draw[xshift=\x cm] (-2,0) -- (-2,0.05);
};
\foreach \x in {0,1,...,8}
{
    \draw[yshift=\x cm] (0,-3) -- (0.05,-3);
};
\tiny
\node[below] at (0.2,0){0};
\foreach \x in {-2,-1}
    \node[below] at(\x,0){\x};
\foreach \y in {1,2,...,9}
    \node[below] at(\y,0){\y};
\foreach \y in {-3,-2,-1}
    \node[left] at(0,\y){\y};
\foreach \y in {1,2,...,4}
    \node[left] at(0,\y){\y};
\footnotesize
\draw[domain=-2.5:-1][color=blue,ultra thick] plot (\x,\x-0.5);
\draw[domain=-1:6][dash pattern=on3pt off3pt][color=blue,ultra thick] plot (\x,0.5*\x-1);
\draw[domain=6:9.5][color=blue,ultra thick] plot (\x,0.5*\x-1) node[right] {$\Phi_{\sigma_{1}}(\lambda)$};
\draw[domain=-2:-0.5][color=red,ultra thick] plot (\x,\x-0.5);
\draw[domain=-0.5:4][color=red,ultra thick] plot (\x,0.666667*\x-0.666667);
\draw[domain=4:6][color=red,ultra thick] plot (\x,2);
\draw[domain=6:9][dash pattern=on3pt off3pt][color=red,ultra thick] plot (\x,2)node[right] {$\Phi_{\sigma_{2}}(\lambda)$};
\draw[domain=2:6][dash pattern=on1pt off1pt][color=blue] plot (\x,0.5*\x-1);

\draw[domain=-2.5:6][dash pattern=on3pt off3pt] plot (\x,\x-0.5) node[left] {$\lambda-\frac{1}{2}$};
\draw[domain=-2.5:7][dash pattern=on3pt off3pt] plot (\x,0.666667*\x-0.666667) node[left] {$\frac{2}{3}\lambda-\frac{2}{3}$};
\draw[domain=-2.5:9][dash pattern=on3pt off3pt] plot (\x,0.5*\x-1) node[left] {$\frac{1}{2}\lambda-1$};
\draw[domain=-2.5:9][dash pattern=on3pt off3pt] plot (\x,2);
\draw[domain=-2.5:9][dash pattern=on3pt off3pt] plot (\x,0.5);
\scriptsize
\node[above]at (0.5,0)[color=orange]{$\lambda_{3}^{[-]}=0.5$};
\draw[dash pattern=on1pt off1pt] (0.5,-0.333333)--(0.5,0);
\node[below]at (1.5,0)[color=orange]{$\lambda_{2}^{[-]}=1.5$};
\draw[dash pattern=on1pt off1pt] (1.5,0)--(1.5,0.333333);
\node[below]at (7.5,0)[color=orange]{$\lambda_1^{[-]}=7.5$};
\draw[dash pattern=on1pt off1pt] (7.5,0)--(7.5,2.75);
\footnotesize
\node[below]at (11,2.3){$\lambda_{1}=8$};
\node[below]at (11,1.8){$\lambda_{2}=2$};
\node[below]at (11,1.3){$\lambda_{3}=1$};

\node[left]at (-0.5,-1){$(-\frac{1}{2},-1)$};
\node[above]at (4,2){$(4,2)$};
\node[below]at (6,2){$(6,2)$};
\end{tikzpicture}
\caption{Application of Newton method to the 
problem given by~\eqref{e:example} \label{f:Newt}}
\end{center}
\end{figure}

At iteration 3, we check that $\lambda_3<\lambda_2$ and find an optimal strategy $\sigma_3$ at $\lambda_{3}^{[-]}=0.5$ with $\sigma_3(1)=2$ and $\sigma_3(2)=2$,  which is the same as at the previous iteration. 
Obviously, we obtain $\lambda_{4}=1=\lambda_3$, so we {\bf stop.} As $\sigma_2=\sigma_3$ is an optimal strategy also at $\lambda_3=\lambda_4=1$, we can find a finite solution using
\begin{equation*}
x=R^{\ast}\otimes((R^{\ast})^{\sharp}\otimes'(\theta\otimes  q\oplus'u))=(-1\; 1)^{\top}
\end{equation*}
and  return $\lambda=1$ and $x=(-1\; 1)^{\top}$ as an optimal solution to the problem.

\subsection{Numerical experiments}
We implemented Algorithms~\ref{a:bisection} and ~\ref{a:Newt-int-opt} in MATLAB and ran some numerical experiments. Before running these experiments we checked the percentage of cases for which $(q^-\otimes p)^{\otimes 1/2}$ (the optimal value of the unconstrained problem) is an optimal value of the constrained problem so that no algorithm is required. Before checking this, the program checked that the system of constraints is feasible and $(q^-\otimes p)^{\otimes 1/2}$ is finite. For $m=n$ ranging from $1$ to $70$ we ran $2000$ experiments for each dimension for the cases where the entries are all finite or we have approximately $30\%$ of finite entries, and where the finite entries are randomly and uniformly selected in the interval $[-500,500]$ or $[-500000,500000]$ (for all four combinations). The results are shown on Figure~\ref{fig:optvalue}. We see that, in general, for very low dimensions it is highly likely that $(q^-\otimes p)^{\otimes 1/2}$, but the probability of this quickly falls to the level of $40\%$ or slightly below, and then starts to grow very slowly. We also observed how this percentage behaves for dimensions up to $400$ performing just $20$ experiments for each dimension and each situation, and recorded that $(q^-\otimes p)^{\otimes 1/2}$ was optimal in $41\%$ of all solvable instances with finite $q^-\otimes p$. 

\begin{figure}[h!]
    \centering
    \begin{tabular}{ll}
    \includegraphics[scale=0.4]{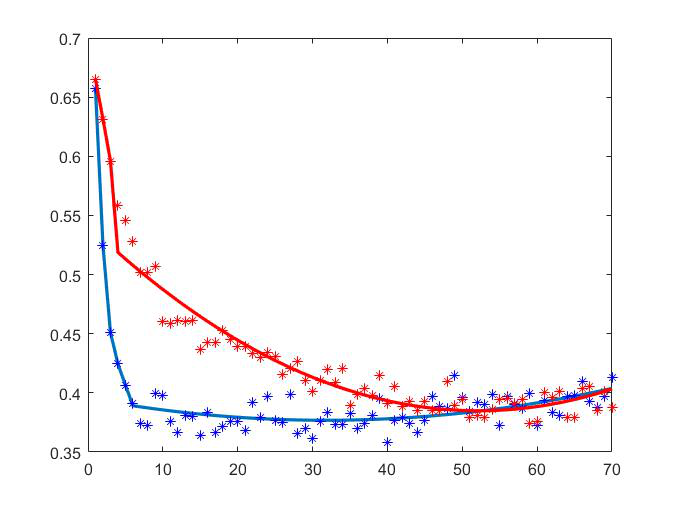}
    &
    \includegraphics[scale=0.4]{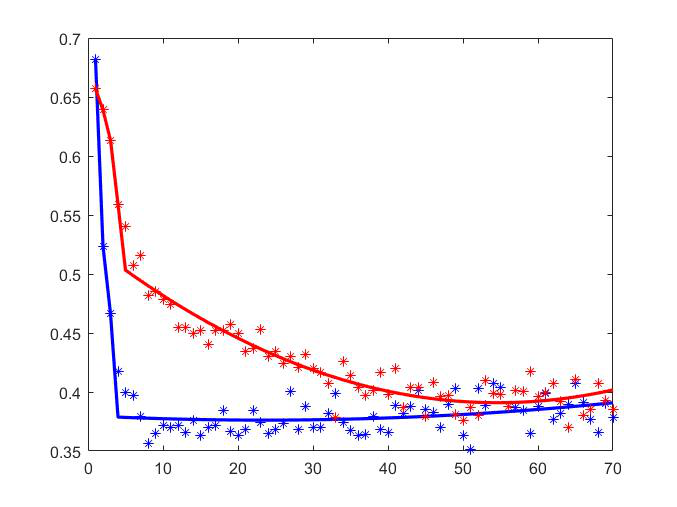}
    \end{tabular}
\caption{Percentage of the cases where $(q^-\otimes p)^{\otimes 1/2}$ is an optimal value for $m=n$ from $1$ to $70$. Finite entries are randomly picked from $[-500,500]$ (left) and from $[-500000,500000]$ (right). Results for the cases where all entries are finite (in blue) are shown versus the cases where approximately $30\%$ of all entries are finite (red).
\label{fig:optvalue}}
\end{figure}

\begin{figure}[h!]
    \centering
\includegraphics[scale=0.5]{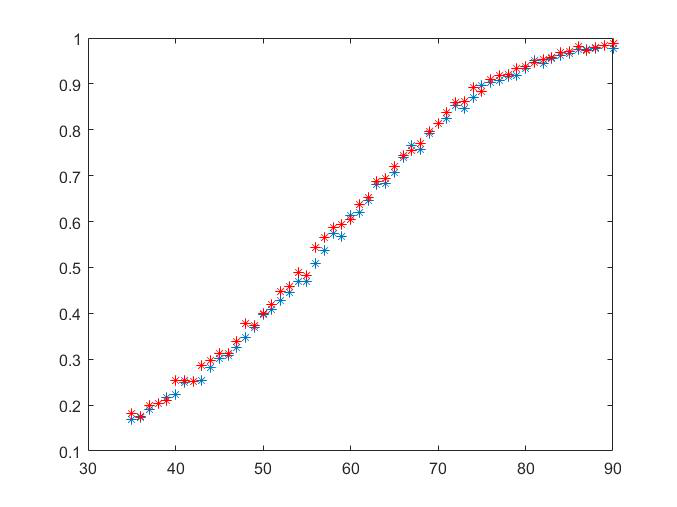}
    \caption{Percentage of the cases where $(q^-\otimes p)^{\otimes 1/2}$ is an optimal value for $m=50$ and $n$ from $35$ to $90$. Blue stars: all entries finite, drawn randomly and uniformly from $[-500,500]$. Red stars: $30\%$ of finite entries, drawn randomly and uniformly from $[-500000,500000]$.
\label{fig:optvaluerect}}
\end{figure}

We also performed similar experiments with rectangular matrices with $50$ rows and number of columns ranging from $35$ to $90$, see Figure~\ref{fig:optvaluerect}. Unlike in the previous experiments, the results do not depend on the range of entries or sparsity. The percentage of cases where $(q^-\otimes p)^{\otimes 1/2}$ is optimal increases monotonically from $15-20\%$ for $n=35$ to almost $100\%$ when $n=90$. Cases $n<35$ were not checked, because for such low $n$ most of the randomly generated constraint systems become infeasible.     

\begin{figure}[h!]
    \centering
    \begin{tabular}{ll}
    \includegraphics[scale=0.4]{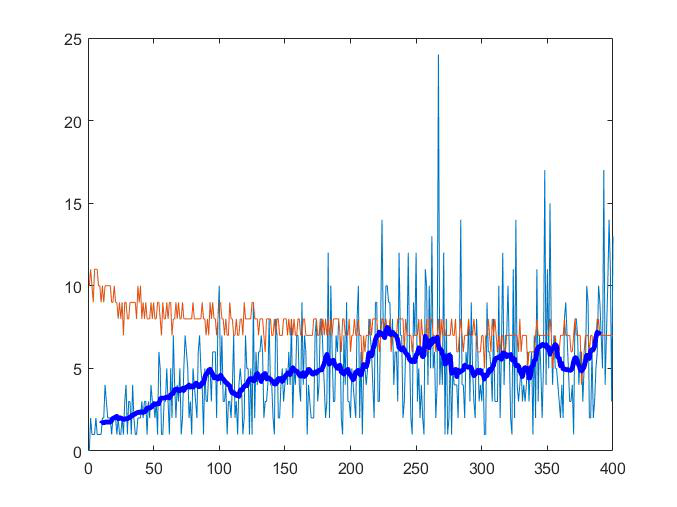}
    &
    \includegraphics[scale=0.4]{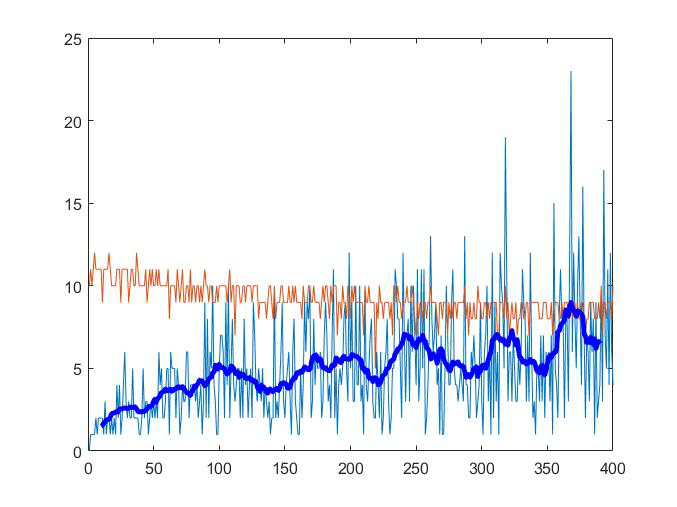}
    \end{tabular}
\caption{Number of iterations of bisection (yellow) and Newton iterations (blue) for randomly generated problems of dimension $1$ to $400$ with the range of entries between $-500$ and $500$ with finite entries only (left) and with approximately $30\%$ of finite entries (right) \label{fig:bisnewt500}}
\end{figure}

\begin{figure}[h!]
    \centering
    \begin{tabular}{ll}
    \includegraphics[scale=0.4]{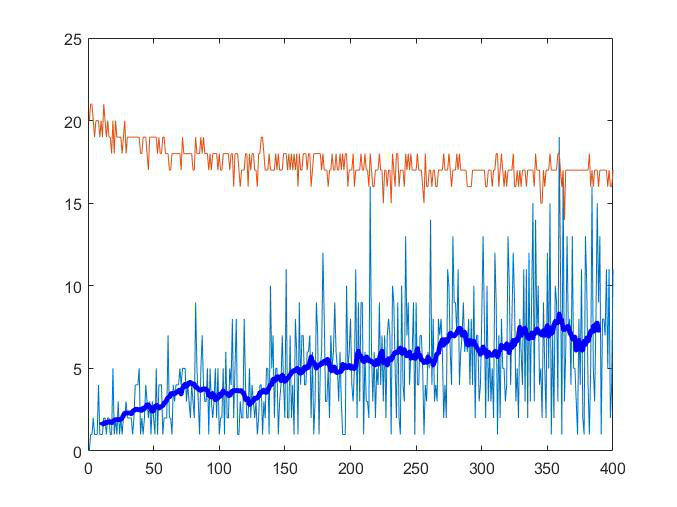}
    &
    \includegraphics[scale=0.4]{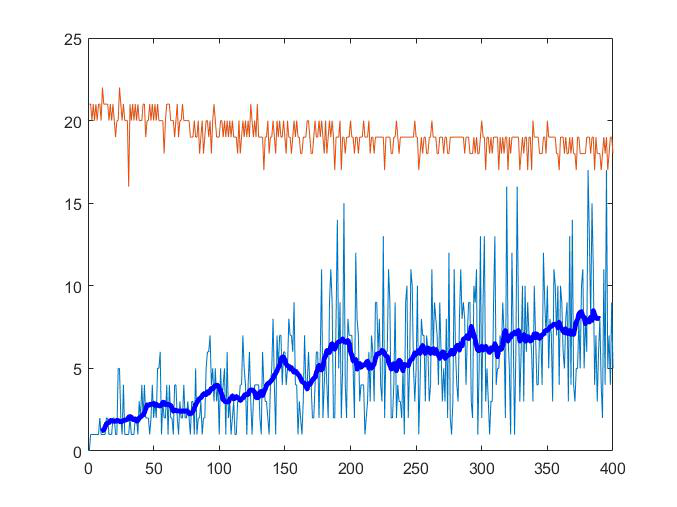}
    \end{tabular}
\caption{Number of iterations of bisection (yellow) and Newton iterations (blue) for randomly generated problems of dimension $1$ to $400$ with the range of entries between $-500000$ and $500000$ with finite entries only (left) and with approximately $30\%$ of finite entries (right) \label{fig:bisnewt500000}}
\end{figure}

Next, we checked the performance of Newton and bisection algorithms, in the same vein as in Gaubert et al.~\cite{BiNewt}. In our experiments, $U$ and $V$ were square matrices with dimensions ranging from $1$ to $400$, with the range of finite entries $[-500,500]$ or $[-500000,500000]$, with finite entries only or with approximately $30\%$ of finite entries.  Results are shown in Figures~\ref{fig:bisnewt500} and~\ref{fig:bisnewt500000}. For each dimension, both methods were run exactly once, after repeatedly ignoring the cases where the system of constraints was infeasible or the value $(q^-\otimes p)^{\otimes 1/2}$ was infinite or optimal for the generated problem.  We see that, as in the case of linear-fractional programming~\cite{BiNewt}, the ``average number'' of Newton iterations (computed here for dimension $d$ as the average taken over dimensions in the range $[d-10,d+10]$ and shown by thick blue line) slowly grows with dimension and, in the case of entries ranging in $[-500,500]$ becomes similar to the number of bisection iterations before the dimension $400$ is reached.  Unlike in the case of linear-fractional programming, here we were able to implement bisection also in the case of sparse data ($30\%$ of finite entries).  Similarly to \cite{BiNewt}, we observed that the number of bisection iterations grows with the range of the finite entries (for an obvious reason, since it means that the interval between lower and upper bounds increases). In particular, the gap between bisection and Newton iterations still remains quite big for $d=400$ for $[-500000,500000]$. However, the number of bisection iterations also decreases with the increase in dimension: this is different from what was observed in tropical linear-fractional programming~\cite{BiNewt} where the number of these iterations was stable.

\section{Pseudoquadratic optimization}
\label{s:pseudoquad}

In this section we will consider pseudoquadratic optimization problem with two-sided constraints: Problem~\ref{prob:main}.
Like Problem~\ref{prob:mainpseudolin}, which we considered earlier, it can be also represented by parametric MPG and solved by means of bisection and Newton methods, calling an MPG solver at each iteration. Below we discuss the details of it, as well as similarity and difference between the pseudoquadratic and pseudolinear tropical optimization.

As a starter, using that $x^-\otimes C\otimes x\leq \lambda$ is equivalent to $C\otimes x\leq \lambda\otimes x$, we can recast Problem~\ref{prob:main} as the problem of finding the least $\lambda$ such that
$A\otimes z\leq B(\lambda)\otimes z$,
is solvable with $z\in\R^{n+1}$, where
\begin{equation}
\label{e:ABlambda-pseudoquad}
A =
    \begin{pmatrix}
    U & b\\
    C & -\infty\\
    -\infty & p\\
    q^- & -\infty \\
    \end{pmatrix} 
    \text{  and  }
    B(\lambda) =
   \begin{pmatrix}
    V & d\\
    \lambda\otimes I & -\infty\\
    \lambda\otimes I & -\infty \\
    -\infty & \lambda\\
    \end{pmatrix}.
\end{equation}
The mean-payoff game diagram corresponding to this system is given on Figure~\ref{f:MPG-pseudoquad}. Comparing it with Figure~\ref{f:MPG-pseudolin}, we see a new group of nodes $[n]$ on top of the diagram, corresponding to the inequalities $C\otimes x\leq\lambda\otimes x$. With respect to this game, we are solving the problem $\min\{\lambda\colon\Phi(\lambda)\geq 0\}$, where $\Phi(\lambda)=\min_i \chi_i(A^{\sharp}B(\lambda))$, with 
$A$ and $B(\lambda)$ defined in~\eqref{e:ABlambda-pseudoquad}. The theory of Section~\ref{ss:recast} applies verbatim.

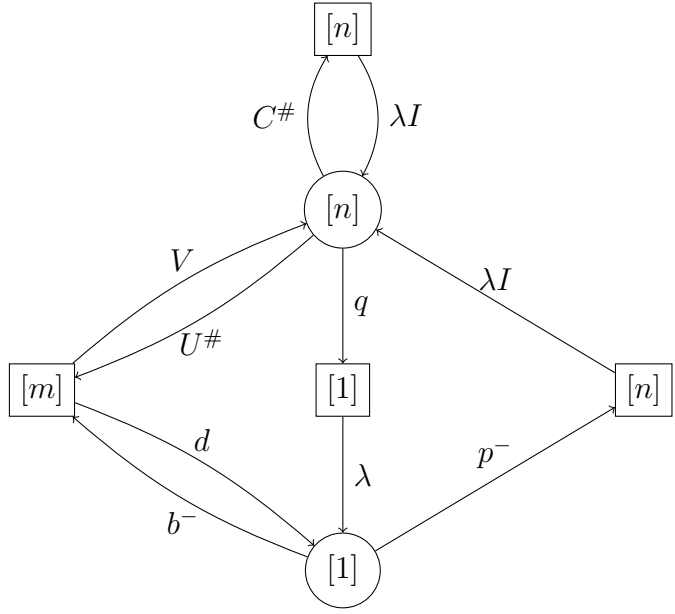
\begin{figure}
\begin{center}
\begin{tikzpicture}[scale=0.8]
    \node[shape=rectangle,draw=black] (m) at (0,3) {$[m]$};
    \node[shape=circle,draw=black] (n) at (5,6) {$[n]$};
    \node[shape=rectangle,draw=black] (nn) at (5,9) {$[n]$};
		\node[shape=rectangle,draw=black] (m+k+1) at (5,3) {[1]};
    \node[shape=circle,draw=black] (n+1) at (5,0) {[1]};
    \node[shape=rectangle,draw=black] (k) at (10,3) {$[n]$};
    
		\path[->] (n) edge[bend left=30] node[left] {$C^{\#}$} (nn);
		\path[->] (nn) edge[bend left=30] node[right] {$\lambda I$} (n);
    \path [->](m) edge[bend left = 10] node[above] {$V$} (n);
    \path [->](n) edge[bend left = 10] node[below] {$U^{\#}$} (m);
    \path [->](m) edge[bend left = 10] node[above] {$d$} (n+1);
    \path [->](n+1) edge[bend left = 10] node[below] {$b^-$} (m);
    \path [->](k) edge node[above] {$\lambda I$} (n);
    \path [->](n+1) edge node[above] {$p^-$} (k);
    \path [->](m+k+1) edge node[right] {$\lambda $} (n+1);
    \path [->](n) edge node[right] {$q$} (m+k+1);
\end{tikzpicture}
\caption{MPG diagram corresponding to pseudoquadratic optimization \label{f:MPG-pseudoquad}}
\end{center}
\end{figure}

It can be also checked that both certificates stated in Propositions~\ref{p:optcert} and ~\ref{p:unbcert} extend to the pseudoquadratic optimization with no change, now referring to the groups of nodes in Figure~\ref{f:MPG-pseudoquad}. However, in the pseudoquadratic programming the cycles in the graph of the game can collect up to $n+1$ repetitions of $\lambda$. Therefore, in the case of integer data, we can only say that the denominator of optimal value is bounded from above by $n+1$. The same is true for any reduced MPG defined by $A$ and $B^{\sigma}(\lambda)$.

An initial upper bound for the bisection method,
following Section~\ref{s:bisection}, can be computed as
$$\lambda_0^{(+)}= x^-\otimes C\otimes x\oplus q^-\otimes x\oplus x^-\otimes p,$$ where $x\in\R^n$ is a solution to $U\otimes x\oplus b\leq V\otimes x\oplus d$. The lower bound comes from the unconstrained pseudoquadratic problem $\min_{x\in\R^n} x^-\otimes C\otimes x\oplus q^-\otimes x\oplus x^-\otimes p$, the optimal value of which was found by Krivulin~\cite{KrivUnconstr-15} to be $$\lambda_0^{(-)}=\rho(C)\oplus (q^-\otimes p)^{\otimes 1/2}.$$
With these initial bounds we can run a usual bisection scheme as an approximate method, or we can still use Algorithm~\ref{a:bisection} as its exact version, but rounding up ($\lceil\cdot\rceil$) now means finding the least rational number greater than or equal to the given number and with denominator bounded by $n+1$. Similarly, rounding down ($\lfloor\cdot\rfloor$) means finding the biggest rational number less than or equal to the given one and with denominator bounded by $n+1$. These operations are still not too difficult: suppose that $\lambda$ is not a integer and set $\lambda=a+\frac{c}{b}$, where $a\in \Z$, $b\in\Nat$, $c\in\Nat\cup\{0\}$ and $\frac{c}{b}$ is an irreducible proper fraction if $c\neq 0$. Then we have
\begin{equation}
\label{e:ceilfloor}
\lceil \lambda\rceil=
\begin{cases}
\lambda, & \text{if $c=0$ or $b\leq n+1$},\\
a+\frac{e_1}{d_1}, & \text{otherwise,}
\end{cases} 
\lfloor \lambda\rfloor=
\begin{cases}
\lambda, & \text{if $c=0$ or $b\leq n+1$},\\
a+\frac{e_2}{d_2}, & \text{otherwise,}
\end{cases}
\end{equation}
where
\begin{equation*}
\frac{e_1}{d_1}=\min\left\{\frac{\lceil\lceil\frac{cd+1}{b}\rceil\rceil}{d}\colon \; 1\leq d\leq n+1\right\},
\end{equation*}
and 
\begin{equation}
\label{e:edfloor}
\frac{e_2}{d_2}=\max\left\{\frac{\lfloor\lfloor\frac{cd-1}{b}\rfloor\rfloor}{d}\colon \; 1\leq d\leq n+1\right\},
\end{equation}
and $\lceil\lceil\cdot\rceil\rceil$ and 
$\lfloor\lfloor\cdot\rfloor\rfloor$, respectively, mean the usual 
operations of rounding up and rounding down to the nearest integers, respectively.

We also have Newton iterations (Algorithm~\ref{a:Newt}) based on left-optimal strategies, which are found using the algebra of germs in general case. In the case of integer data,
the denominators of $\lambda_k$
do not exceed $n+1$ and the denominators of breakpoints of $\Phi(\lambda)$ do not exceed $(n+1)^2$. Hence the largest possible denominator of the distance between $\lambda_k$ and a breakpoint does not exceed $(n+1)^3$, so we can take $\epsilon=1/(n+1)^3$ to ensure that a strategy that is optimal at $\lambda-\epsilon$ is optimal for the whole interval $[\lambda-\epsilon,\lambda]$.

Integer version of Newton iterations 
(Algorithm~\ref{a:Newt-int-opt}) also works, but in this algorithm $\lambda^{[-]}$ should be defined as the largest rational number that is 1) strictly smaller than $\lambda$, 2) has a denominator bounded by $n+1$. 
Following the notation in~\eqref{e:ceilfloor}, we can obtain:
\begin{equation*}
\lambda^{[-]}=
\begin{cases}
\lambda-\frac{1}{n+1}, & \text{if $\lambda$ is integer}, \\
a+\frac{e_2}{d_2}, & \text{otherwise},
\end{cases}
\end{equation*}
where $\frac{e_2}{d_2}$ is defined as in~\eqref{e:edfloor}. 

Solution of~\eqref{e:subproblem} is an important ingredient in any modification of Newton algorithm. This is treated exactly as in Section~\ref{sss:partial}, since $C\otimes x\leq\lambda\otimes x$ does not add to the available strategies of Max, and they are still determined by the right-hand side of the system $U\otimes x\oplus b\leq V\otimes x\oplus d$. 
Proposition~\ref{p:sysreduced} then leads us to a problem of the following type:
\begin{equation*}
\begin{split}
& \min_{x\in\R^n} x^-\otimes p \oplus q^-\otimes x\oplus x^-\otimes C\otimes x \\
& \text{s.t.}\quad l \leq x \leq u\ \text{and}\  
R\otimes x \leq x.
\end{split}
\end{equation*}
An explicit expression for the optimal value of this problem and solution set was obtained in Krivulin~\cite{KrivOurVolume-14}[Theorem 4] and it can be used instead of~\eqref{e:thetalambdak} and~\eqref{e:x}. However, the computational complexity of computing the optimal value of this problem rises to $O(n^5)$.


\appendix

\section{Proof of Proposition~\ref{p:constrpseudolin}}

We first represent Problem~\ref{prob:constrpseudolin} equivalently as
\begin{equation}
 \min_{x\in\R^n,\lambda\in\R}\{\lambda\colon\ \   x^-\otimes  p \oplus q^-\otimes  x\leq \lambda,\quad g \leq x \leq h,\quad U\otimes x \leq x\}.
\label{MPGEx0}
\end{equation}

By multiplying the inequality $U\otimes x \leq x$ by $U$, we deduce that $U^{\otimes 2}\otimes x \leq U\otimes x \leq x$, and hence, by continuing to multiply the inequality by $U$ we determine that $U^*\otimes x \leq x$. Conversely, $U^*\otimes x \leq x$ implies that $U\otimes x \leq x$. Hence, we can restate \eqref{MPGEx0} as follows
\begin{equation}
\min_{x\in\R^n ,\lambda\in\R}\{\lambda\colon
\ \  p \leq \lambda\otimes  x,\quad q^-\otimes  x \leq \lambda,\quad 
g \leq x,\quad h^-\otimes x \leq 0,\quad U^*\otimes x \leq x\}.
\label{MPGEx1}
\end{equation}

Figure~\ref{f:constrainedMPG} provides a mean-payoff game representation of this problem. Here, the group of $n$ nodes of Min (in a circle) corresponds to
$n$ variables and the lone-standing node of Min (circle in the bottom) corresponds to the free column. There are two individual nodes of Max corresponding to $q^-\otimes x\leq \lambda$ and $h^-\otimes x\leq 0$.  The remaining $3n$ nodes of Max are split into three groups of $n$ nodes: 1) the group on the top ($U^*\otimes x\leq x$), 2) the group on the left $(g\leq x)$ and 3) the group on the right ($p\leq\lambda\otimes x$). It is also agreed that an arc between two nodes exists if and only if the corresponding entry of the vector or the matrix marking the corresponding group of arcs on the diagram is finite.

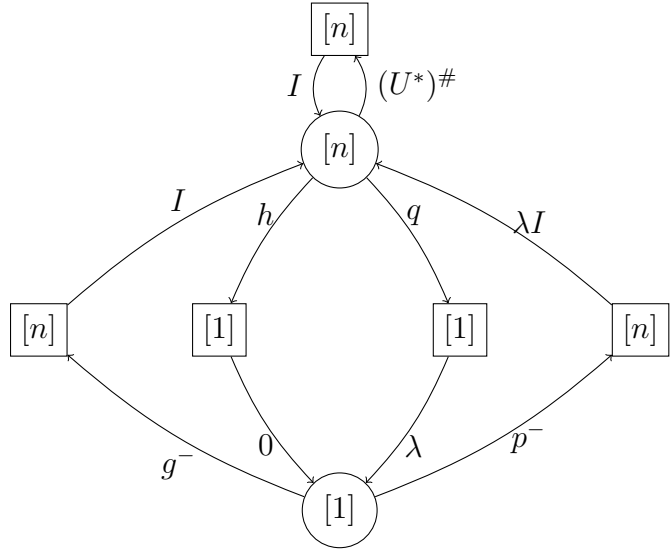
\begin{figure}
\begin{center}
    \begin{tikzpicture}[scale=0.8]
    \node[shape=rectangle,draw=black] (n_1) at (0,3) {$[n]$};
    \node[shape=circle,draw=black] (n) at (5,6) {$[n]$};
    \node[shape=rectangle, draw=black] (n_3) at (5,8) {$[n]$};
    \node[shape=rectangle,draw=black] (1_a) at (3,3) {$[1]$};
    \node[shape=rectangle,draw=black] (1_b) at (7,3) {$[1]$};
    \node[shape=circle,draw=black] (1) at (5,0) {$[1]$};
    \node[shape=rectangle,draw=black] (n_2) at (10,3) {$[n]$};
    
    \path [->](n_1) edge[bend left = 10] node[above] {$I$} (n);
    \path [->](1) edge[bend left = 10] node[below] {$g^-$} (n_1);
    \path [->](n) edge[bend right = 10] node[above] {$h$} (1_a);
    \path [->](n) edge[bend left = 10] node[above] {$q$} (1_b);
    \path [->](1_a) edge[bend right = 10] node[below] {$0$} (1);
    \path [->](1_b) edge[bend left = 10] node[below] {$\lambda$} (1);
    \path [->](1) edge[bend right = 10] node[right] {$p^-$} (n_2);
    \path [->](n_2) edge[bend right = 10] node[right] {$\lambda I$} (n);
    \path [->](n) edge[bend right = 30] node[right] {$(U^*)^\#$} (n_3);
    \path [->](n_3) edge[bend right = 30] node[left] {$I$} (n);
    \end{tikzpicture}
\caption{The parametric mean-payoff game corresponding to \eqref{MPGEx1}} 
\label{f:constrainedMPG}
\end{center}
\end{figure}

Observe that at every node at which player Max is active, there is no choice to make as 
there is only one arc leaving that node. Also, let $j'$ be the node (of one of the $[n]$ groups of nodes of Max 
on the left and on the right) chosen by Min at the bottom node (corresponding to the free column). 
Assuming this choice of Min and examining the mean-payoff game diagram we see that the total weight of any cycle can be written as 
$s_{i_1 i_2} + ... +s_{i_k i_1}$, where 
\begin{equation}
     s_{ij} \in 
    \begin{cases}
    \{-u_{ji}^*\} & \text{if } j \neq j' \\
    \{-u_{j'i}^*, q_i - p_{j'} + 2\lambda, q_i + \lambda - g_{j'}, h_i + \lambda - p_{j'}, h_i - g_j'\} & \text{if  } j = j',
    \end{cases}
\label{e:scases}
\end{equation}
where the above possibilities for $s_{ij}$ are valid only if all matrix and vector entries that take part in them are finite. Hence we obtain that the optimal value of \eqref{MPGEx1} is equal to the least value of $\lambda$ such that
\begin{equation}
    0 \leq \underset{j'}{\text{min}} \;   \underset{i_1,...,i_k}{\text{min}} \; \{s_{i_1 i_2} + ... s_{i_k i_1}\}
\label{e:minimise}    
\end{equation}
 where $s_{ij}$ can take the values described in~\eqref{e:scases}.

As at every node of player Max there is no choice,
 we can delete these nodes and aggregate the weights. Also, we can omit the cycles whose weights are composed entirely from the entries of $(U^*)^{\sharp}$, as these cycle weights do not depend on  $\lambda$ and are nonnegative. We are then left with the cycles that go through the node of Min in the bottom of the diagram and hence also through $j'$. Note that, as node $j'$ can appear in a cycle only once, there will be at most one occurrence of the second case, where all terms except for the first one come from the arcs in the lower part of the diagram (going to node $n+1$ of Min and back). All other $s_{ij}$ can be compressed to an entry of $(U^*)^{\sharp}$ by using inequality
\begin{equation*}
    u_{i_1i_2}^* + ... + u_{i_{k-1}i_k}^* \leq u_{i_1i_k}^*,
\end{equation*}
valid since $(U^*)^{\otimes (k-1)} = U^*$ for each $k > 1$. Thus we have to consider all $3$-cycles of the following form:
\begin{center}
    \begin{tikzpicture}
    \node[shape=circle,draw=black] (i) at (5,4) {$i$};
    \node[shape=circle,draw=black] (j) at (5,0) {$j'$};
    \node[shape=circle,draw=black] (1) at (10,2) {1};
    
    \path [->](i) edge node[above] {$w_1$} (1);
    \path [->](1) edge node[below] {$w_2$} (j);
    \path [->](j) edge node[left] {$ -u_{ij'}^*$} (i);
\end{tikzpicture}
\end{center}
Here $w_1 + w_2$ is one of the finite values in $\{g_i - p_{j'} + 2 \lambda,\;
    q_i + \lambda - g_{j'},\;
    h_i + \lambda - p_{j'},\; h_i-g_{j'}\}$, and $i$ is such that $u^*_{ij'}$ and one of these values are finite.
    
So we need to find the minimal $\lambda$ such that:
\begin{itemize}
\item[1.] $-u_{ij'}^* + q_i - p_{j'} + 2\lambda \geq 0$ for all $i$ and $j'$ such that $u_{ij'}^*$, $q_i$ and $p_{j'}$ are finite: this is equivalent to $\lambda \geq (q^-\otimes U^*\otimes p)^{\otimes \frac{1}{2}}$;
\item[2.] $-u_{ij'}^* + q_i + \lambda - g_{j'} \geq 0$ for all $i$ and $j'$ such that $u_{ij'}^*$, $q_i$ and $g_{j'}$ are finite: equivalent to $\lambda \geq q^-\otimes U^*\otimes g$;
\item[3.] $-u_{ij'}^* + h_i + \lambda -p_{j'} \geq 0$
for all $i$ and $j'$ such that $u_{ij'}^*$, $h_i$ and $p_{j'}$ are finite: equivalent to
$\lambda \geq h^-\otimes U^*\otimes p$;
\item[4.] $-u_{ij'}^* + h_i - g_{j'} \geq 0$ for all $i$ and $j'$ such that $u_{ij'}^*$, $h_i$ and $g_{j'}$ are finite: this is always satisfied by the problem assumption $U^*\otimes g\leq h$. 
\end{itemize}
 Hence the optimal value of $\lambda$ is equal to
$(q^-\otimes U^*\otimes p)^{\otimes \frac{1}{2}} \oplus h^-\otimes U^*\otimes p \oplus q^-\otimes U^*\otimes g$, as claimed.

We now deduce the representation of solution set (this part of the proof is similar to that of Krivulin~\cite{KrivPseudolin-14}, Theorem 6). The solution set is given by the same inequalities as in~\eqref{MPGEx0} but for the optimal value $\lambda=\theta$, so it is the (finite part of the) alcoved polyhedron described by the following inequalities:
\begin{equation}
\label{e:optsystem}
\begin{split}
& x^-\otimes p\leq\theta,\quad q^-\otimes x\leq \theta\\    
& g\leq x, \quad x\leq h,\\
& U\otimes x\leq x
\end{split}    
\end{equation}
The first inequality can be rewritten as $p\leq \theta\otimes x$ and the second inequality can be rewritten as $x\leq \theta\otimes q$, and then the first four inequalities of~\eqref{e:optsystem} can be merged into
\begin{equation}
\label{e:intineqs}    
g \oplus \theta^-\otimes p \leq x \leq \theta\otimes q \oplus' h. 
\end{equation}
It remains to prove that a finite $x$ satisfies~\eqref{e:intineqs} and $U\otimes x\leq x$ if and only if $x$ is as in~\eqref{e:solset-pseudolin}. 

Assume first that a finite $x$ satisfies~\eqref{e:intineqs} and $U\otimes x\leq x$. The latter inequality is equivalent to $U^*\otimes x=x$, and therefore from the right-hand side of~\eqref{e:intineqs} we have 
$U^*\otimes x\leq \theta\otimes q\oplus' h$ and, 
using Proposition~\ref{p:equiv}, 
$x\leq (U^*)^{\sharp}\otimes' (\theta\otimes q\oplus' h)$. As we also have $x\geq g\oplus\theta^-\otimes p$ from the left-hand side of~\eqref{e:intineqs}, we obtain that $x=U^*\otimes v$,
where $v=x$ satisfies
\begin{equation*}
g\oplus \theta^-\otimes p\leq x=v\leq (U^*)^{\sharp}\otimes' (\theta\otimes q\oplus' h),    
\end{equation*}
hence $x$ is in~\eqref{e:solset-pseudolin}.

Now assume that $x$ is as in~\eqref{e:solset-pseudolin}. Since $x=A^*\otimes v$ and $A\otimes A^*\leq A^*$, it satisfies $A\otimes x\leq x$. Since $v$ is finite, $x=U^*\otimes v$ is also finite. 
We also have $g\oplus \theta^-\otimes p\leq x$ since $U^*\otimes v\geq v$, and 
$x=U^*\otimes v\leq \theta\otimes q\oplus' h$ follows by Proposition~\ref{p:equiv}.

\section{Proofs of optimality and unboundedness certificates}
The proofs written below work both in pseudolinear and in pseudoquadratic case.

\subsection{Proof of Proposition~\ref{p:optcert}}
$\lambda^{*}$ is optimal if and only if $\Phi(\lambda^{*})=0$ and $\Phi(\lambda)<0$ for all $\lambda<\lambda^{*}$. By the left-hand side of~\eqref{PhiEquals}, $\Phi(\cdot)$ is a pointwise minimum of a finite number of continuous non-decreasing functions $\Phi_{\tau}(\cdot)$. Therefore, the above property implies that $\lambda^*$ is optimal if and only if there exists a strategy $\tau$ of Min such that $\Phi_{\tau}(\lambda^{*})=0$ and $\Phi_{\tau}(\lambda)<0$ for all $\lambda<\lambda^{*}$. We can further find $j$
and small enough $\epsilon$ such that   
$\Phi_{\tau}(\lambda)=\chi_j(A^{\sharp}_{\tau}B(\lambda))
=a\lambda+b$ 
for all $\lambda\in[\lambda^*-\epsilon,\lambda]$ and some $a>0$. 

Since $\chi_{j}(A^{\sharp}_{\tau}B(\lambda))$ is the maximal cycle mean (per turn) over all the cycles accessible from the node $j$ of Min in the reduced game defined by $A_{\tau}$ and $B$, having $\Phi_{\tau}(\lambda^{*})=0$ means that in the mean-payoff game defined by $A_{\tau}$ and $B(\lambda^{*})$ all cycles accessible from the node $j$ of Min have non-positive weight and at least one of them has zero weight. Furthermore, $\chi_j(A^{\sharp}_{\tau}B(\lambda))=a\lambda+b$ with some $a>0$ and $b$ for all $\lambda\in[\lambda^{*}-\epsilon, \lambda^{*}]$ if and only if all cycles of zero weight accessible from the node $j$ of Min contain one of the nodes of player Max that is not from the $[m]$ group on the left of Figure~\ref{f:MPG-pseudolin} or Figure~\ref{f:MPG-pseudoquad}. Thus we have deduced that existence of a strategy $\tau$ with the claimed properties is equivalent to the optimality of $\lambda^*$.

\subsection{Proof of Proposition~\ref{p:unbcert}}

The problem is unbounded if and only if $\Phi(\lambda)\geq 0$ for all $\lambda\in\R$. Since $\Phi(\lambda)=\max_{\sigma}\Phi^{\sigma}(\lambda)$, this condition is equivalent to the following:
\begin{equation}
\label{e:condhl1}
\forall\lambda\in\R\mbox{ } \exists\sigma\mbox{ }\text{s.t.}\mbox{ } \Phi^{\sigma}(\lambda)\geq 0,    
\end{equation}
which can occur if and only if  there exist $\sigma$ and $\lambda'\in\mathbb{R}$ such that $\Phi^{\sigma}(\lambda)=a$ for all $\lambda\leq\lambda'$ and some constant $a\geq 0$ since function $\Phi^{\sigma}(\lambda)$ is non-decreasing, piecewise-linear and continuous, which also yields that such $\sigma$ can make $\Phi(\lambda)\geq 0$ hold for all $\lambda\in\mathbb{R}$. Thus condition~\eqref{e:condhl1} can equivalently be written as:
\begin{equation*}
\exists\sigma\mbox{ }\text{s.t.} \mbox{ }\Phi^{\sigma}(\lambda)\geq 0 \mbox{ }\forall\lambda\in\R.
\end{equation*}
Now we use that $\Phi^{\sigma}(\lambda)=\min\limits_{i}\chi_i(A^{\sharp}B^{\sigma}(\lambda))=\min\limits_{i}\min\limits_{\tau}\Phi_{A,B(\lambda)}(i,\tau,\sigma)$ to rewrite the above condition as follows:
\begin{equation}
\label{e:condhl3}
\exists\sigma \mbox{ }\text{s.t.} \mbox{ }\forall i,\mbox{ }\forall\tau \mbox{ }\Phi_{A,B(\lambda)}(i,\tau,\sigma)\geq 0\mbox{ } \forall\lambda\in\mathbb{R},
\end{equation}
which is equivalent to saying that the weights of all cycles in
the digraph defined by $A$ and $B^{\sigma}(\lambda)$ are nonnegative and cannot contain $\lambda$. Obviously, this condition does not depend on the value of $\lambda$, which can be set to $0$ as in the claim.

Note that the weight of a cycle must contain $\lambda$ if and only if such cycle contains a node of Max, which is not from the $[m]$ group (on the left of Figure~\ref{f:MPG-pseudolin} or~\ref{f:MPG-pseudoquad}) since any other node of Max has unique outgoing arc with weight $\lambda$. Therefore, condition~\eqref{e:condhl3} holds if and only if any cycle is avoiding these nodes of Max and all cycles have a nonnegative weight.

\end{document}